\documentclass[12pt]{amsart}

\usepackage{caption}
\usepackage{subcaption}
\usepackage{quoting}
\quotingsetup{leftmargin=.5cm, rightmargin=.5cm}

\usepackage{overpic}

\usepackage[dvipsnames]{xcolor}

\usepackage{emptypage} 

\makeindex 

\pagebreak
\usepackage{thm-restate}
\newtheorem{thm}{Theorem}

\newtheorem{prop}[thm]{Proposition}
\newtheorem{lem}[thm]{Lemma}

\newtheorem{conj}[thm]{Conjecture}
\theoremstyle{definition}

\newtheorem{defn}[thm]{Definition}
\newtheorem{remk}[thm]{Remark}
\newtheorem{remks}[thm]{Remarks}

\newtheorem{exm}[thm]{Example}
\newtheorem{exms}[thm]{Examples}
\newtheorem{notat}[thm]{Notation}
\newtheorem{question}[thm]{Question}

{\hfill$\square$\end{defn}}
{\hfill$\square$\end{remk}}
{\hfill$\square$\end{remks}}
{\hfill$\square$\end{exm}}
{\hfill$\square$\end{exms}}
{\hfill$\square$\end{notat}}

\newcommand{\RR}{\mathbb{R}}      
\newcommand{\ZZ}{\mathbb{Z}}      

\renewcommand{\ker}{{\rm{ker}}}

\newcommand{\tuborg}{\left\{\begin{array}{ll}}
\newcommand{\sluttuborg}{\end{array}\right.}

\newcounter{elno}

\newcounter{elno-abc}   

\newcounter{elno-abc-prime}   

\usepackage[english]{babel}
\usepackage[utf8x]{inputenc}
\usepackage{amsmath}
\usepackage{graphicx}
\usepackage[colorinlistoftodos]{todonotes}

\title{Tight contact structures on toroidal plumbed 3-manifolds}
\author{Tanushree Shah}
\email{tanushrees@cmi.ac.in}
\author{Jonathan Simone}
\email{jsimone@andrew.cmu.edu}
 \thanks{\emph{Subjclass[2020]}: 57K10, 57K14, 57K33 }
\begin{document}

\begin{abstract}
 We consider tight contact structures on plumbed 3-manifolds with no bad vertices. We discuss how one can count the number of tight contact structures with zero Giroux torsion on such 3-manifolds and explore conditions under which Giroux torsion can be added to these tight contact structures without making them overtwisted. We give an explicit algorithm to construct stein diagrams corresponding to tight structures without Giroux torsion. We focus mainly on plumbed 3-manifolds whose vertices have valence at most 3 and then briefly consider the situation for plumbed 3-manifolds with vertices of higher valence.
\end{abstract}

\maketitle
\section{Introduction} 
\ 

The classification of tight contact structures up to isotopy is known only for a few classes of 3-manifolds. Eliashberg classified tight contact structures on $S^3$, $\mathbb{R}^3$ and $S^2\times S^1$ in \cite{E3}. Kanda \cite{Kan} and Giroux \cite{Gir1} (independently) gave classifications on 3-torus. Etnyre \cite{Et1} classified tight contact structures on some lens spaces. Honda gave a complete classification of tight contact structures on lens spaces, solid tori, and toric annuli with convex boundary in \cite{H1} and on torus bundles which fiber over the circle, and circle bundles which fiber over closed surfaces in \cite{H2}. There is a partial classification for small Seifert fibered spaces \cite{W} \cite{GLS, GLS1}, \cite{Mat}. \ 
\

In \cite{CGH}, it was shown that every atoroidal 3-manifold admits finitely many tight contact structures. In \cite{HKM1}, it is shown that if a 3-manifold has an incompressible torus, however, then it admits infinitely many contact structures. These contact structures come in infinite families resulting from the addition of \textit{Giroux torsion} in a neighborhood of the incompressible torus. We will make this notion more precise below.

Giroux torsion has been studied for relatively few families of contact 3-manifolds---surface bundles over $S^1$ \cite{H2}, certain plumbed 3-manifolds \cite{Sim}, and nonloose torus knot complements \cite{EMM}. An interesting aspect of adding Giroux torsion to a tight contact structure is that it does not always preserve tightness. That is, given a tight contact 3-manifold with an incompressible torus $T$, the contact structure obtained by adding Giroux torsion in a neighborhood of $T$ may or may not be tight. The above-mentioned papers have shed some light on this phenomenon and it appears that the tool of convex surface theory can help us approach the following question.

\begin{question}
Let $T$ be a convex incompressible torus in a tight contact 3-manifold $(Y,\xi)$. Under what conditions does adding Giroux torsion in a neighborhood of $T$ preserve tightness?
\label{q:GT}
\end{question}

For the earliest toriodal contact 3-manifolds studied---e.g. $T^3$ \cite{Kan}, $T^2\times I$ \cite{H1}, and surface bundles over $S^1$ \cite{H2}---the addition of Giroux torsion could only preserve tightness when added to universally tight contact structures.
More recent papers (\cite{EMM},\cite{Sim}) have shown that Giroux torsion can also be added to virtually overtwisted contact structures while preserving tightness. These papers explored 3-manifolds containing a single incompressible torus.

Of course, a contact 3-manifold can contain many incompressible tori. 
If two incompressible tori are isotopic, then adding Giroux torsion to one torus provides the same contact structure (up to isotopy) as adding the same amount of Giroux torsion to the other torus. For example, in the case of $T^2$-bundles over $S^1$ (see \cite{H2}), there are many isotopic incompressible tori; however, we view the 3-manifold as having only one incompressible torus along which Giroux torsion can be added. 
If two nonisotopic incompressible tori intersect nontrivially, then Giroux torsion can be added to only one of the tori; that is, Giroux torsion cannot be added to both tori simultaneously. This occurs in the case of Seifert fibered spaces with four singular fibers (see \cite{Sha}).

We start by exploring tight contact structures on 3-manifolds containing only incompressible tori that are disjoint. To this end, we consider plumbed 3-manifolds whose graphs have no bad vertices\footnote{the negative of the weight of a given vertex is greater than or equal to the valence of the vertex} and whose vertices are most trivalent. See Figure \ref{fig:example} for an example of such a 3-manifold\footnote{The edges of each cycle in a plumbing graph must be decorated with either $+$ or $-$ to indicate how the plumbing operation is to be performed. Undecorated edges are assumed to be positive. We address this in Section \ref{sec:stein}}. 
Such manifolds contain a single isotopy class of incompressible tori for each linear path connecting two trivalent vertices; for example, the plumbed 3-manifold in Figure \ref{fig:example} contains seven nonisotopic incompressible tori.

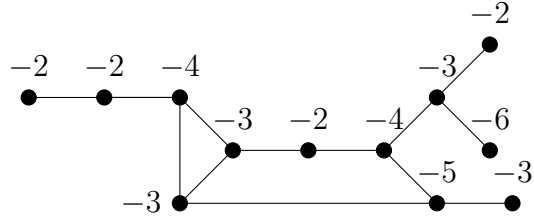
\begin{figure}
\centering 
\begin{tikzpicture}[dot/.style = {circle, fill, minimum size=1pt, inner sep=0pt, outer sep=0pt}]
\tikzstyle{smallnode}=[circle, inner sep=0mm, outer sep=0mm, minimum size=2mm, draw=black, fill=black];
\node[smallnode, label={$-3$}] (a) at (0,0) {};
\node[smallnode, label={$-2$}] (b) at (1,0) {};
\node[smallnode, label={$-4$}] (c) at (2,0) {};

\node[smallnode, label={$-4$}] (a1) at (-.7,.7) {};
\node[smallnode, label=left:{$-3$}] (a2) at (-.7,-.7) {};
\node[smallnode, label={$-3$}] (c1) at (2.7,.7) {};
\node[smallnode, label={$-5$}] (c2) at (2.7,-.7) {};

\node[smallnode, label={$-3$}] (c21) at (3.7,-.7) {};

\node[smallnode, label={$-2$}] (c11) at (3.4,1.4) {};
\node[smallnode, label={$-6$}] (c12) at (3.4,0) {};

\node[smallnode, label={$-2$}] (a11) at (-1.7,.7) {};
\node[smallnode, label={$-2$}] (a12) at (-2.7,.7) {};

\draw[-] (a) -- (b);
\draw[-] (b) -- (c);
\draw[-] (a) -- (a1);
\draw[-] (a) -- (a2);
\draw[-] (a1) -- (a2);
\draw[-] (c) -- (c1);
\draw[-] (c) -- (c2);
\draw[-] (a2) -- (c2);
\draw[-] (a1) -- (a11);
\draw[-] (a1) -- (a12);
\draw[-] (c2) -- (c21);
\draw[-] (c1) -- (c11);
\draw[-] (c1) -- (c12);

\end{tikzpicture}
\caption{A plumbed 3-manifold}\label{fig:example}
\end{figure}

Before discussing our results, we will first recall the definition of Giroux torsion and the more general notation of \textit{twisting}. 
Let $T\in Y$ be a incompressible torus. We say that $(Y,\xi)$ has $\textit{$\frac{m}{2}-$twisting}$ in a neighborhood of $T$ if there exists a contact embedding of ($T^2\times I, \xi_n =\ker(\sin(m\pi z)dx+\cos(m\pi z)dy)$) into $(Y,\xi)$ such that $T^2\times\{t\}$ are isotopic to $T$ \cite{GH}. 
$(Y,\xi)$ is called \textit{minimally twisting} if it does not have $\frac{m}{2}-$twisting for all $m\ge1$ in a neighborhood of any incompressible torus. Finally for $n\in\mathbb{Z}$, we say that $Y$ has $n-$Giroux torsion in a neighborhood of $T$, if it has $n-$twisting in a neighborhood of $T$. 
We are now ready to state our results.

\begin{thm} Let $M$ be a plumbed 3-manifold with no bad vertices and whose vertices are all at most trivalent. Let the weights of the vertices be $-a_1,\ldots,-a_n$. Then $M$ admits at least $(a_1-1)\cdots(a_n-1)$ Stein fillable contact structures.
\label{thm:mainminimallytwisting}
\end{thm}

\begin{remk}
    To prove Theorem 2, we will provide an algorithm one can use to ``wrap up" the plumbing diagram so that the methods of \cite{GoS} can be used to draw a handlebody diagram of the associated 4-dimensional plumbing (whose boundary is the original plumbed 3-manifold). We then describe how to algorithmically turn this diagram into a Stein diagram via handleslides and stabilization. 
\end{remk}

Since Stein fillable contact structures have no Giroux torsion by \cite{Gay}, Theorem \ref{thm:mainminimallytwisting} gives a lower bound on the number of tight contact structures with no Giroux torsion.
We next use convex surface theory to obtain upper bounds on the number of tight contact structures with prescribed twisting. Determining a general upper bound for any such plumbed 3-manifold is rather involved. Instead, we produce upper bounds for one family of plumbed 3-manifolds; the techniques we use to obtain this upper bound can be used to handle other cases on an ad hoc basis.

\begin{thm} Let $Y$ be the plumbed 3-manifold shown in Figure \ref{fig:example1} and let $m\in\ZZ_{\ge1}$. Then $Y$ admits at most:
\begin{enumerate}
 \item $\displaystyle(a_1-1)\cdots(a_n-1)\prod_{i=1}^4(b_1^i-1)\cdots(b^i_{k_i}-1)$
minimally twisting tight contact structures; and
 \item $\displaystyle2\prod_{i=1}^4(b_2^i-1)\cdots(b^i_{k_i}-1)$
tight contact structures with $\frac{m}{2}$-twisting.
\end{enumerate}
    \label{thm:mainex}\label{thm:mainGT}
\end{thm}

\begin{remk}
Although the authors believe that the bound in Theorem \ref{thm:mainGT} is sharp, they have been unable to prove it using the current technology. In principle, Honda's state transversal argument should be able to be used to prove that these contact structures are tight, and then further means would be needed to show that they are distinct. 
\label{rem:sharp}
\end{remk}


\begin{remk}
By Theorem \ref{thm:mainminimallytwisting}, the 3-manifold $Y$ in Figure \ref{fig:example1} has a least $\displaystyle(a_1-1)\cdots(a_n-1)\prod_{i=1}^4(b_1^i-1)\cdots(b^i_{k_i}-1)$ tight contact structures with zero Giroux torsion. This matches the upper bound of the number of minimally twisting tight contact structures on $Y$ given by Theorem \ref{thm:mainGT}. However, there are (potentially) $\displaystyle2\prod_{i=1}^4(b_2^i-1)\cdots(b^i_{k_i}-1)$ more tight contact structures with zero Giroux torsion (i.e. the contact structures with $\frac{1}{2}-$twisting). The authors suspect that the contact structures obtained from Theorem \ref{thm:mainminimallytwisting} provide the minimally twisting contact structures given in Theorem \ref{thm:mainGT}. It is worth noting, however, that contact structures induced from Stein structures on a fixed 4-manifold need not have the same amount of twisting. For example, in \cite{Sha}, the Stein diagrams obtained from the obvious star-shaped plumbing diagrams of Seifert fibered spaces with four singular fibers provide both minimally twisting tight contact structures and tight contact structures with $\frac{1}{2}-$twisting.
\end{remk}

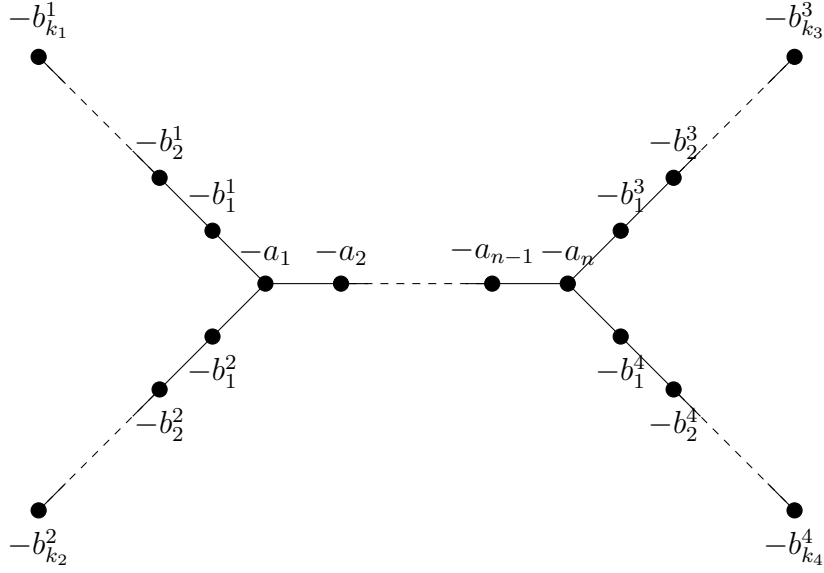
\begin{figure}
\centering 
\begin{tikzpicture}[dot/.style = {circle, fill, minimum size=1pt, inner sep=0pt, outer sep=0pt}]
\tikzstyle{smallnode}=[circle, inner sep=0mm, outer sep=0mm, minimum size=2mm, draw=black, fill=black];

\node[smallnode, label={$-a_1$}] (a) at (0,0) {};
\node[smallnode, label={$-a_2$}] (b) at (1,0) {};
\node[smallnode, label={$-a_{n-1}$}] (e) at (3,0) {};
\node[smallnode, label={$-a_n$}] (f) at (4,0) {};
\node[] (c) at (1.5,0) {};
\node[] (d) at (2.5,0) {}; 

\node[smallnode, label={$-b_1^1$}] (a11) at (-.7,.7) {};
\node[smallnode, label={$-b_2^1$}] (a12) at (-1.4,1.4) {};
\node[] (aa) at (-1.9,1.9) {};
\node[] (ab) at (-2.5,2.5) {}; 
\node[smallnode, label={$-b_{k_1}^1$}] (a13) at (-3,3) {};

\node[smallnode, label=below:{$-b_1^2$}] (b11) at (-.7,-.7) {};
\node[smallnode, label=below:{$-b_2^2$}] (b12) at (-1.4,-1.4) {};
\node[] (ba) at (-1.9,-1.9) {};
\node[] (bb) at (-2.5,-2.5) {}; 
\node[smallnode, label=below:{$-b_{k_2}^2$}] (b13) at (-3,-3) {};

\node[smallnode, label={$-b_1^3$}] (c11) at (4.7,.7) {};
\node[smallnode, label={$-b_2^3$}] (c12) at (5.4,1.4) {};
\node[] (ca) at (5.9,1.9) {};
\node[] (cb) at (6.5,2.5) {}; 
\node[smallnode, label={$-b_{k_3}^3$}] (c13) at (7,3) {};

\node[smallnode, label=below:{$-b_1^4$}] (d11) at (4.7,-.7) {};
\node[smallnode, label=below:{$-b_2^4$}] (d12) at (5.4,-1.4) {};
\node[] (da) at (5.9,-1.9) {};
\node[] (db) at (6.5,-2.5) {}; 
\node[smallnode, label=below:{$-b_{k_4}^4$}] (d13) at (7,-3) {};

\draw[-] (a) -- (b);
\draw[-] (b) -- (c);
\draw[dashed] (b) -- (e);
\draw[-] (d) -- (e);
\draw[-] (e) -- (f);

\draw[-] (a) -- (a11);
\draw[-] (a) -- (a2);
\draw[-] (f) -- (c11);
\draw[-] (f) -- (d11);

\draw[-] (a11) -- (a12);
\draw[-] (a12) -- (aa);
\draw[dashed] (a12) -- (a13);
\draw[-] (ab) -- (a13);

\draw[-] (b11) -- (b12);
\draw[-] (b12) -- (ba);
\draw[dashed] (b12) -- (b13);
\draw[-] (bb) -- (b13);

\draw[-] (c11) -- (c12);
\draw[-] (c12) -- (ca);
\draw[dashed] (c12) -- (c13);
\draw[-] (cb) -- (c13);

\draw[-] (d11) -- (d12);
\draw[-] (d12) -- (da);
\draw[dashed] (d12) -- (d13);
\draw[-] (db) -- (d13);

\end{tikzpicture}
\caption{Plumbed 3-manifold $Y$}\label{fig:example1}
\end{figure}

The proof of Theorem 3 partly addresses Question \ref{q:GT} for the family of 3-manifolds $Y$ shown in Figure \ref{fig:example1}. In particular, we will see that if $(Y,\xi)$ has no Giroux torsion and it does not satisfy a \textit{rigidity} requirement (defined in Section \ref{sec:prelim}), then adding Giroux torsion to $(Y,\xi)$ creates an overtwisted contact structure. Hence the only contact structures with zero Giroux torsion to which Giroux torsion can be added while preserving tightness are the contact structures satisfying the rigidity requirement. 

\subsection{Higher Valence}
We have thus far considered plumbed 3-manifolds that only contain disjoint incompressible tori (up to isotopy). In Section \ref{nfibers}, we will briefly discuss the case of plumbed 3-manifolds containing vertices of valence $\ge4$, which necessarily contain intersecting, nonisotopic incompressible tori.

\subsection{A Conjecture}
It follows from the work of Honda \cite{H1} that if a tight contact 3-manifold $(Y,\xi')$ is obtained from $(Y,\xi)$ by adding Giroux torsion in a neighborhood of an incompressible torus $T$, then any toric annulus $T^2\times I$ with convex boundary containing $T$ must have the property that $(T^2\times I,\xi|_{T^2\times I})$ is universally tight. We have also seen that $(Y,\xi)$ itself need not be universally tight. This suggests the following. 

\begin{conj}
   Let $T$ be a convex incompressible torus in a tight contact 3-manifold $(Y,\xi)$ such that every toric annulus neighborhood of $T$ with convex boundary is universally tight. Then the contact 3-manifold obtained by adding Giroux torsion in a neighborhood of $T$ is tight.
\end{conj}

\subsection{Organization} The paper is organized as follows. In Section \ref{sec:prelim}, we will recall important facts related to convex surface theory and define the notion of \textit{rigidity}. In Section \ref{sec:blocks}, we will explore contact structures on $\Sigma\times S^1$ and so-called $\textit{maximal chains}$, which will be used in proving Theorem \ref{thm:mainGT}. We then prove Theorem \ref{thm:mainGT} in Section \ref{sec:example}. In Section \ref{sec:stein}, we prove Theorem \ref{thm:mainminimallytwisting} by constructing Stein diagrams. Finally, in Section \ref{nfibers} we discuss contact structures on general plumbed 3-manifolds with no bad vertices.

\subsection{Acknowledgements} The authors would like to thank John Etnyre for many helpful conversations.
The first author was partially supported by the Infosys Fellowship.

\section{Rigidity}\label{sec:prelim}

We assume the reader is familiar with convex surface theory due to Giroux \cite{Gir2} and bypass attachments and edge rounding due to Honda \cite{H1}. For a nice exposition on the basics of convex surface
theory, see \cite{GS}. Here we will recall some notation and important results regarding toric annuli and basic slices that we must gather in order to discuss the notion of \textit{rigidity}. 

For a convex surface $\Sigma$ in a contact 3-manifold, we denote the set of dividing curves by $\Gamma_\Sigma$ and the slope of the dividing curve by $s(\Gamma_\Sigma)$.
Let us consider a tight contact structure $\xi$ on $T^2\times I$. Fix an identification $T^2=\mathbb{R}^2/\mathbb{Z}^2$. Let $s(\Gamma_{T×\{i\}})=s_i$ for $i=0,1$. 
$(T^2\times I,\xi)$ is called a $\textit{basic slice}$ if: $\xi$ is tight; $T_{i}$ are convex and $\#\Gamma_{T_{i}}=2$, for $i=0,1$; the minimal integral representatives of $\mathbb{Z}^{2}$ corresponding to $s_{i}$ (for $i=0,1$) form a $\mathbb{Z}$-basis of $\mathbb{Z}^{2}$; and every convex torus parallel to the boundary has slope between $s_0$ and $s_1$.
After a diffeomorphism of $T^{2}$, we may assume that a basic slice has $s(\Gamma_{(T^2)\times\{1\}})=-1$ and $s(\Gamma_{(T^2)\times\{0\}})=0$.
By \cite{H1}, a basic slice can have two tight contact structures up to isotopy, differentiated by the sign of their relative Euler classes. We will call these positive and negative basic slices.


Given any tight $(T^2\times I,\xi)$, there exists a natural grouping of the basic slice layers into blocks via continued fractions. These blocks are special because the basic slices that are part of the same continued fraction block can be ``shuffled" without changing the contact structures. In \cite{H1}, it is shown that shuffling the basic slices within a given continued fraction block does not change the contact structure. See \cite{H1} for details.

We now are now ready to define rigidity. Consider $\Sigma\times S^1$, where $\Sigma$ is a pair of pants. Identify each boundary component of $\Sigma\times S^1$ with $\RR^2/\ZZ^2$ by choosing $(1,0)^T$ to be the direction given by $-\partial(\Sigma_1\times S^1)$ and $(0,1)^T$ to be the direction given by the $S^1-$fiber. Let $-\partial(\Sigma\times S^1)=T_0+T_1+T_2$. Note that our orientation convention differs from the orientation convention in \cite{H1}.
Suppose $s(\Gamma_{T_i})=\infty$ for all $i$. Glue toric annuli $T_i\times[0,1]$ to $\Sigma\times S^1$ such that $T_i\times\{1\}=T_i$ for $i=0,1,2$. We say that the pair $T_i, T_j$ is \textit{rigid} if the basic slices of $T_i\times[0,1]$ cannot be shuffled so that the innermost basic slices have opposite sign; otherwise we call the pair \textit{nonrigid}.

\begin{remk}
The notion of rigidity can be viewed as a generalization of the notion of \textit{totally 2-inconsistency} as defined and used in \cite{EMM}.
\end{remk}

We now develop conditions under which nonrigid pairs give rise to overtwisted contact structures. This is the main tool we will use to obstruct the tightness of contact structures obtained by adding Giroux torsion.

\begin{lem}
Let $Y=\Sigma\times S^1\cup T_0\times[0,1]\cup T_1\times[0,1]\cup T_2\times[0,1]$ be as above and assume $s(\Gamma_{T_0\times\{0\}})\le1$ and $s(\Gamma_{T_i\times\{0\}})\ge-1$ for $i=1,2$. If $T_i, T_j$ is nonrigid, then there exists a thickening of $T_k\times [0,1]$ (where $k\neq i,j$) to a toric annulus $T_k\times[0,2]\subset Y$ such that $T_k\times[0,2]$ is not minimally twisting.
\label{lem:addtwist}
\end{lem}

\begin{proof}
 Assume that $T_1, T_2$ is nonrigid; the other cases are analogous. 
 Since $s(\Gamma_{T_i\times\{0\}})\ge -1$ and $s(\Gamma_{T_i\times\{1\}})=s(\Gamma_{T_i})=\infty$ for $i=1,2$, there exists a torus $T_i\times\{t\}$ in between $T_i$ and $T_i\times\{0\}$ with slope $s(\Gamma_{T_i\times\{t_i\}})=-1$ such that $T_i\times[t_i,1]$ is a basic slice. By assumption, we can perform shuffling to ensure that $T_1\times[t_1,1]$ and $T_2\times[t_2,1]$ have the same sign. By Lemma 4.13 in \cite{GS}, there exists a vertical annulus from a Legendrian ruling of $T_1\times\{t_1\}$ to a Legendrian ruling of $T_2\times\{t_2\}$ without boundary parallel boundary curves. Thus we can edge round to obtain a convex torus parallel to $T_0$ with boundary slope $1$. The result follows.
\end{proof}

\begin{prop}[c.f. \cite{EMM}] Let $Y=\Sigma\times S^1\cup T_0\times[0,1]\cup T_1\times[0,1]\cup T_2\times[0,1]$ be as in Lemma \ref{lem:addtwist}.
\begin{enumerate}
    \item Let $Y'=Y\cup_{T_0} D^2\times S^1$. If $T_1, T_2$ is a nonrigid pair, then $Y'$ is overtwisted.
    \item $Y'=Y\cup_{T_0} D^2\times S^1\cup_{T_1} D^2\times S^1$. If $T_i, T_2$ is a nonrigid pair for $i\in\{1,2\}$, then $Y'$ is overtwisted.
    \item $Y'=Y\cup_{T_0} D^2\times S^1\cup_{T_1} D^2\times S^1 \cup_{T_2} D^2\times S^1$. If $T_i, T_j$ is a nonrigid pair for any $i\neq j\in\{0,1,2\}$, then $Y'$ is overtwisted.
\end{enumerate}
    \label{prop:rigid}
\end{prop}

\begin{proof}
(1) Since $T_1, T_2$ is a nonrigid pair, by Lemma \ref{lem:addtwist}, there exists a nonminimally twisting toric annulus parallel to $T_0$. Hence there  exists a torus parallel to $T_0$ whose dividing curve maps to the meridian of $\partial (D^2\times S^1)$. It follows that $Y$ is overtwisted.

(2) and (3) now follow from (1). Indeed, we can remove the appropriate $D^2\times S^1s$ to obtain the manifold under consideration in part (1), which is overtwisted.
\end{proof}

\section{Maximal Chains}\label{sec:blocks}

If $Y$ is a plumbed 3-manifold whose vertices have valence at most 3, then it can be decomposed into various building blocks. In this section we will discuss \textit{maximal chains} (defined below). 
Before discussing maximal chains, we first prove the following fact about the boundary slopes of $\Sigma\times S^1$, where each boundary component is identified with $\RR^2/\ZZ^2$ as in Section \ref{sec:prelim}.

\begin{lem} Let $-\partial(\Sigma\times S^1)=T_0+T_1+T_2$.
Suppose $s(\Gamma_{T_0})\ge1$, $s(\Gamma_{T_i})\ge -1$ and $\#\Gamma_{T_i}=2$ for $i=1,2$. Then there exists a convex torus $\tilde{T}_i$ parallel to $T_i$ for $i=1,2$ with $s(\Gamma_{\tilde{T_i}})=-1$ and $\#\Gamma_{\tilde{T_i}}=2$, and a convex torus $\tilde{T}_0$ parallel to $T_0$ with $s(\Gamma_{\tilde{T}_0})=1$ and $\#\Gamma_{\tilde{T_0}}=2$.
\label{lem:sT_0=1}
\end{lem}

\begin{proof}
Let $A$ be a vertical annulus from a Legendrian ruling of $T_1$ to a Legendrian ruling of $T_2$. Iteratively add bypasses in $A$ to either $T_1$ or $T_2$ until there are no more bypasses. Let $\overline{T}_i$ denote the resulting torus parallel to $T_i$; note that $s(\Gamma_{\overline{T}_i})\le s(\Gamma_{T_i})$ and $\#\Gamma_{\overline{T}_i}=2$ or $i=1,2$. 

Note that either $s(\Gamma_{\overline{T}_i})=-1$ for $i=1,2$ or $s(\Gamma_{\overline{T}_i})=\infty$ for $i=1,2$. Assume the latter case first. Let $A$ be a vertical annulus from a Legendrian ruling of $T_0$ to a Legendrian divide of $\overline{T}_1$. Then we may iteratively add bypasses along $T_0$ until we obtain a torus $\overline{T}_0$ parallel to $T_0$ with boundary slope $\infty$ and two dividing curves. Hence there exist tori $\tilde{T}_i$ between $T_i$ and $\overline{T}_i$ with $s(\Gamma_{\tilde{T}_0})=1$, $s(\Gamma_{\tilde{T}_i})=-1$ for $i=1,2$, and $\#\Gamma_{\tilde{T}_i}=2$ for all $i$.

Now assume that $s(\Gamma_{\overline{T}_i})\ge-1$ for $i=1,2$; then
since any vertical annulus between $\overline{T}_1$ and $\overline{T}_2$ only has two dividing curves and no bypasses, edge rounding using $\overline{T}_i$ and $A$ yields a torus $\overline{T}_0$ parallel to $T_0$ with boundary slope $s(\Gamma_{\overline{T}_0})\le 1$; it follows that there exists a convex torus $\tilde{T}_0$ in between $T_0$ and $\overline{T}_0$ with boundary slope $1$ and $\#\Gamma_{\tilde{T_0}}=2$. 
Next take a vertical annulus between $\tilde{T}_0$ and $\overline{T}_1$. Once again, add all possible bypasses in a vertical annulus between the Legendrian rulings. Note that if we obtain tori with boundary slope $\infty$, then we may proceed as above. Assume this is not the case so that we obtain a torus $\tilde{T}_1$ parallel to $\overline{T}_1$ boundary slope $-1$.
We can then edge round using $\tilde{T}_1$ and $\tilde{T}_0$ to find a torus $\tilde{T}_2$ parallel to $\overline{T}_2$ with boundary slope $-1$ and two dividing curves. 
\end{proof}

Let $L=L(p,q)$ and $$\frac{p}{q}=[a_1,\ldots,a_n] \,\,\,\text{ and }\,\,\, \frac{p'}{q'}=[a_1,\ldots,a_{n-1}]$$ be negative continued fraction expansions.
$L$ has linear plumbing diagram with weights $(-a_1,\ldots,a_n)$.
$L$ can be constructed by gluing two copies of $D^2\times S^1$. Identify $\partial(D^2\times S^1)=S^1\times S^1$ with $\RR^2/\ZZ^2$ by choosing $(1,0)^T$ to be the direction given by the meridian and let $(0,1)^T$ be given by the longitudinal direction. Then by \cite{N} the gluing map (after switching the factors of $S^1\times D^2$) is given by 
$$\begin{bmatrix}-p' & -q'\\p&q \end{bmatrix}.$$

Take a pair of disjoint simple closed curves of the form $pt\times S^1$ in each copy of $D^2\times S^1$ and remove open neighborhoods of these curves. The result is diffeomorphic to $\Sigma\times S^1$, where $\Sigma$ is a pair of pants. Denote the first copy by $\Sigma_1\times S^1$ and the second copy by $\Sigma_2$. Identify each boundary torus of $\Sigma_1\times S^1$ with $\RR^2/\ZZ^2$ by choosing $(1,0)^T$ to be the direction given by $-\partial(\Sigma_1\times S^1)$ and $(0,1)^T$ to be the direction given by the $S^1-$fiber. Note that this orientation agrees with the orientation chosen on the original copies of $D^2\times S^1$. 
Let $-\partial(\Sigma_i\times S^1)=T^i_0+T^i_1+T^i_2$, where $T^i_0$ is the boundary of the original $D^2\times S^1$. 
Then the gluing map 
$g:T_0^1\to T_0^2$ is given by 
$$\begin{bmatrix}-p' & q'\\-p&q \end{bmatrix}.$$
The resulting 3-manifold $C$ is a plumbed 3-manifold with four boundary components, which we draw schematically in Figure \ref{fig:maxchainlocal}.
If $C$ is embedded in an ambient plumbed 3-manifold, we call $C$ a maximal chain.
We aim to understand the boundary slopes of maximal chains. Since maximal chains are built out of pairs of pants, we first show the following.

\begin{figure}
\centering 
\begin{tikzpicture}[dot/.style = {circle, fill, minimum size=1pt, inner sep=0pt, outer sep=0pt}]
\tikzstyle{smallnode}=[circle, inner sep=0mm, outer sep=0mm, minimum size=2mm, draw=black, fill=black];
\node[smallnode, label={$-a_1$}] (a) at (0,0) {};
\node[smallnode, label={$-a_2$}] (b) at (1,0) {};
\node[smallnode, label={$-a_{n-1}$}] (e) at (3,0) {};
\node[smallnode, label={$-a_n$}] (f) at (4,0) {};

\node[] (c) at (1.5,0) {};
\node[] (d) at (2.5,0) {}; 
\node[smallnode, draw=black, fill=white] (a1) at (-.7,.7) {};
\node[smallnode, draw=black, fill=white] (a2) at (-.7,-.7) {};
\node[smallnode, draw=black, fill=white] (f1) at (4.7,.7) {};
\node[smallnode, draw=black, fill=white] (f2) at (4.7,-.7) {};

\draw[-] (a) -- (b);
\draw[-] (b) -- (c);
\draw[dashed] (b) -- (e);
\draw[-] (d) -- (e);
\draw[-] (e) -- (f);

\draw[-] (a) -- (a1);
\draw[-] (a) -- (a2);
\draw[-] (f) -- (f1);
\draw[-] (f) -- (f2);

\end{tikzpicture}
\caption{Maximal Chain in $Y$}\label{fig:maxchainlocal}
\end{figure}
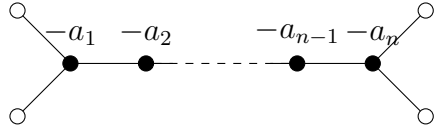

\begin{lem}
Suppose $s(\Gamma_{T_i^j})\ge-1$ and $\#\Gamma_{T_i^j}=1$ for all $i,j\in\{1,2\}$. Then for each $i,j$, there exists a convex torus $\tilde{T}_i^j$ parallel to $T_i^j$ with boundary slope $s(\Gamma_{\tilde{T}_i^j})=-1$. Moreover, there exists disjoint convex tori $\tilde{T}_0^j$ (for $j=1,2$) parallel to $T_0^j$ with boundary slope $s(\Gamma_{\tilde{T}_0^j})=1$ .
\label{lem:maxchainupperbound}
\end{lem}

\begin{proof}
By Lemma \ref{lem:sT_0=1}, it suffices to show that there exists a torus $\tilde{T}_0^j$ parallel $T_0^j$ with boundary slope at least 1 for $j=1,2$.
Let $A_j$ be a vertical annulus from a Legendrian ruling of $T_1^j$ to a Legendrian ruling of $T_2^j$. Iteratively attach all possible bypasses in $A_j$ along $\tilde{T}_1^j$ and $\tilde{T}_2^j$ until there are no further bypasses; note that when this process terminates, the boundary slope of the resulting tori $\hat{T}_i^j$ is $s(\Gamma_{\hat{T}_i^j})\ge-1$. 
Edge rounding yields a torus $\hat{T}_0^j$ parallel to $T_0^j$ with slope $s(\Gamma_{\hat{T}_0^j})\le 1$. Let $s(\Gamma_{\hat{T}_0^1})=\frac{a}{b}.$

We will now recut $Y$ along a couple of times; this procedure is drawn schematically in Figure \ref{fig:recutY}.
First, recut $Y$ along $\hat{T}_0^1$ and glue the resulting toric annulus $X$ to $\Sigma_2\times S^1$ along their common boundary (see second picture in Figure \ref{fig:recutY}). 
This provides a thickened copy of $\Sigma_2\times S^1$ with boundary components $T_1^2$, $T_2^2$, and $\hat{T}_0^2$, where $\hat{T}_0^2$ is the image of $\hat{T}_0^1$ under the gluing map.
Then $s(\Gamma_{\hat{T}_0^2})=\frac{bp-aq}{bp'-aq'}\ge1$. It follows that there exists a convex torus between $T_0^2$ and $\hat{T}_0^2$, which we denote by $\tilde{T}_0^2$ with slope $1$.
Once again, recut $Y$ along $\tilde{T}_0^2$ (see the third picture in Figure \ref{fig:recutY}). By a similar argument, there exist a convex torus $\tilde{T}_0^1$ parallel to $T_0^1$ with boundary slope 1. 
\end{proof}

\begin{figure}
\centering
\begin{overpic}
[scale=.8]{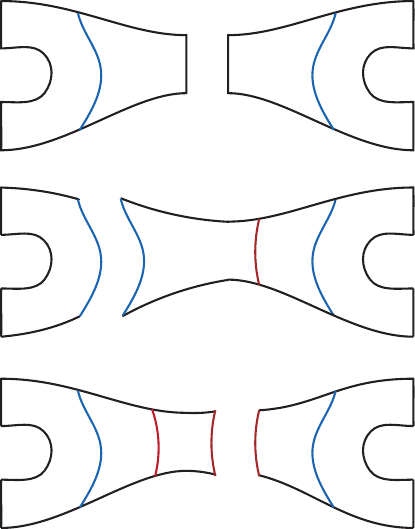}
\put(-7,73.5){$\hat{T}_1^1$}
\put(-7,93){$\hat{T}_2^1$}
\put(-7,38){$\hat{T}_1^1$}
\put(-7,57){$\hat{T}_2^1$}
\put(-7,3){$\hat{T}_1^1$}
\put(-7,22){$\hat{T}_2^1$}

\put(80,73.5){$\hat{T}_1^2$}
\put(80,93){$\hat{T}_2^2$}
\put(80,38){$\hat{T}_1^2$}
\put(80,57){$\hat{T}_2^2$}
\put(80,3){$\hat{T}_1^2$}
\put(80,22){$\hat{T}_2^2$}

\put(32,75){$T_0^1$}
\put(42,75){$T_0^2$}
\put(15,68){{\color{teal}$\hat{T}_0^1$}}
\put(60,68){{\color{teal}$\hat{T}_0^2$}}

\put(15,33){{\color{teal}$\hat{T}_0^1$}}
\put(60,33){{\color{teal}$\hat{T}_0^2$}}
\put(46,38){{\color{purple}$\tilde{T}_0^2$}}

\put(46,3){{\color{purple}$\tilde{T}_0^2$}}
\put(27,2.5){{\color{purple}$\tilde{T}_0^1$}}

\end{overpic}
\caption{Recutting $Y$}\label{fig:recutY}
\end{figure}
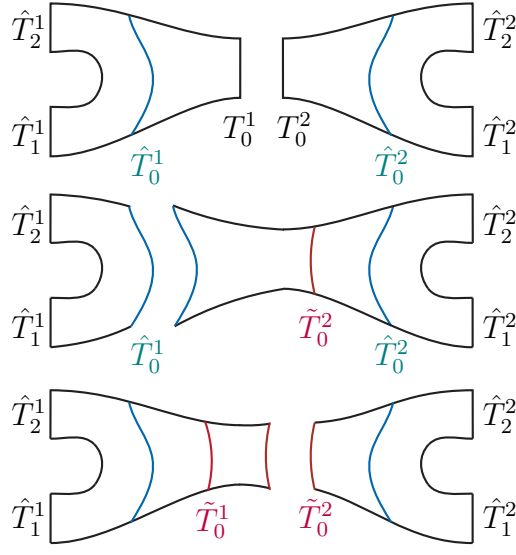

\section{An Example}\label{sec:example}

Let $Y$ be the plumbing 3-manifold shown in Figure \ref{fig:plumbing}. 

\begin{alignat*}{2}
& \frac{p}{q}=[a_1,\ldots,a_n] \qquad && \frac{p'}{q'}=[a_1,\ldots,a_{n-1}]\\
& \frac{x_i}{y_i}=[b_1^i,\ldots,b_{k_i}^i] \qquad && \frac{x_i'}{y_i'}=[b_1^i,\ldots,b_{k_i-1}^i]
\end{alignat*}

Decompose $Y$ into the gluing of four solid tori $V_1,\ldots,V_4$ and the maximal chain $Y'$. Let $Y'=(\Sigma_1\times S^1)\cup_g( \Sigma_2\times S^1)$, where $g:T_0^1\to T_0^2$ is given by 
$$\begin{bmatrix}-p' & q'\\-p&q \end{bmatrix}.$$
Let $-\partial (\Sigma_1\times S^1)=T_0^1+T_1+T_2$ and $-\partial (\Sigma_2\times S^1)=T_0^1+T_3+T_4$
Then the solid torus $V_i$ for $i\in\{1,2,3,4\}$ is glued to $\Sigma_1\times S^1$ via the map $g:\partial V_i\to T_i$ given by 
$$\begin{bmatrix}x_i & x_i'\\-y_i&-y_i' \end{bmatrix}.$$

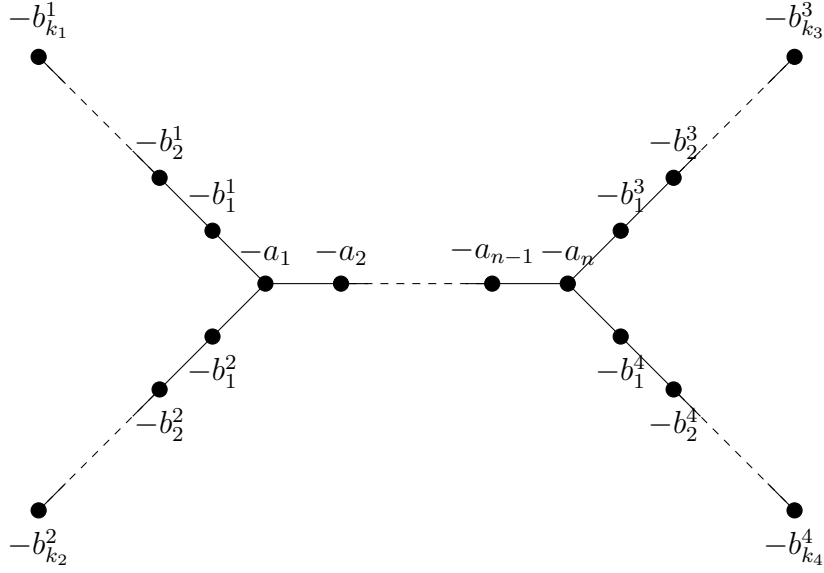
\begin{figure}
\centering 
\begin{tikzpicture}[dot/.style = {circle, fill, minimum size=1pt, inner sep=0pt, outer sep=0pt}]
\tikzstyle{smallnode}=[circle, inner sep=0mm, outer sep=0mm, minimum size=2mm, draw=black, fill=black];

\node[smallnode, label={$-a_1$}] (a) at (0,0) {};
\node[smallnode, label={$-a_2$}] (b) at (1,0) {};
\node[smallnode, label={$-a_{n-1}$}] (e) at (3,0) {};
\node[smallnode, label={$-a_n$}] (f) at (4,0) {};
\node[] (c) at (1.5,0) {};
\node[] (d) at (2.5,0) {}; 

\node[smallnode, label={$-b_1^1$}] (a11) at (-.7,.7) {};
\node[smallnode, label={$-b_2^1$}] (a12) at (-1.4,1.4) {};
\node[] (aa) at (-1.9,1.9) {};
\node[] (ab) at (-2.5,2.5) {}; 
\node[smallnode, label={$-b_{k_1}^1$}] (a13) at (-3,3) {};

\node[smallnode, label=below:{$-b_1^2$}] (b11) at (-.7,-.7) {};
\node[smallnode, label=below:{$-b_2^2$}] (b12) at (-1.4,-1.4) {};
\node[] (ba) at (-1.9,-1.9) {};
\node[] (bb) at (-2.5,-2.5) {}; 
\node[smallnode, label=below:{$-b_{k_2}^2$}] (b13) at (-3,-3) {};

\node[smallnode, label={$-b_1^3$}] (c11) at (4.7,.7) {};
\node[smallnode, label={$-b_2^3$}] (c12) at (5.4,1.4) {};
\node[] (ca) at (5.9,1.9) {};
\node[] (cb) at (6.5,2.5) {}; 
\node[smallnode, label={$-b_{k_3}^3$}] (c13) at (7,3) {};

\node[smallnode, label=below:{$-b_1^4$}] (d11) at (4.7,-.7) {};
\node[smallnode, label=below:{$-b_2^4$}] (d12) at (5.4,-1.4) {};
\node[] (da) at (5.9,-1.9) {};
\node[] (db) at (6.5,-2.5) {}; 
\node[smallnode, label=below:{$-b_{k_4}^4$}] (d13) at (7,-3) {};

\draw[-] (a) -- (b);
\draw[-] (b) -- (c);
\draw[dashed] (b) -- (e);
\draw[-] (d) -- (e);
\draw[-] (e) -- (f);

\draw[-] (a) -- (a11);
\draw[-] (a) -- (a2);
\draw[-] (f) -- (f1);
\draw[-] (f) -- (f2);

\draw[-] (a11) -- (a12);
\draw[-] (a12) -- (aa);
\draw[dashed] (a12) -- (a13);
\draw[-] (ab) -- (a13);

\draw[-] (b11) -- (b12);
\draw[-] (b12) -- (ba);
\draw[dashed] (b12) -- (b13);
\draw[-] (bb) -- (b13);

\draw[-] (c11) -- (c12);
\draw[-] (c12) -- (ca);
\draw[dashed] (c12) -- (c13);
\draw[-] (cb) -- (c13);

\draw[-] (d11) -- (d12);
\draw[-] (d12) -- (da);
\draw[dashed] (d12) -- (d13);
\draw[-] (db) -- (d13);

\end{tikzpicture}
\caption{Plumbed 3-manifold $Y$}\label{fig:plumbing}
\end{figure}

We may isotope the core of $V_i$ so that it has very negative twisting number $-m<<0$. Then we may take $V_i$ to be a standard tubular neighborhood of the core with boundary slope $s(\Gamma_{\partial V_i})=-\frac{1}{m}$ and two dividing curves. It follows that $-1<s(\Gamma_{T_i})=-\frac{my_i-y'_i}{mx_i-x'_i}<0$. By Lemma \ref{lem:maxchainupperbound}, there exists a convex torus parallel to $T_i$ with two dividing curves and boundary slope $-1$. Recut $Y$ along this torus and label it $T_i$; hence $s(\Gamma_{T_i})=-1$ for all $i$. Moreover, by Lemma \ref{lem:maxchainupperbound}, there exist a torus $\tilde{T}_0^i$ parallel to $T_0^i$ with two dividing curves and boundary slope $1$. Recut $Y$ along $T_0^1$. Then the boundary slopes of $\Sigma_1\times S^1$ are $-1,-1,1$ and the boundary slopes of $\Sigma_2\times S^1$ are $-1,-1,\frac{p-q}{p'-q'}$; moreover, there exists a toric annulus $T^2\times[0,1]\subset \Sigma_2\times S^1$ such that $T^2\times\{0\}=T_0^2$ and $T^2\times\{1\}=\tilde{T}_0^2$. Let $\Sigma_2'\times S^1$ be such that $\Sigma_2\times S^1=(\Sigma_2'\times S^1)\cup_{\tilde{T}_0^2}T^2\times[0,1]$.

\begin{lem}
If $\Sigma_1\times S^1$ or $\Sigma_2\times S^1$ contains a vertical Legendrian with twisting number 0, then $(Y,\xi)$ is not minimally twisting.
\label{lem:twisting}
\end{lem}

\begin{proof}
    Without loss of generality, assume $\Sigma_1\times S^1$ has a vertical Legendrian $\gamma$. Let $T_0:=T_0^1$. Then we can take vertical annuli from $\gamma$ to a Legendrian ruling of $T_i$ for all $i$. Then we can attach bypasses to each torus until we obtain parallel tori $\hat{T}_0^1, \hat{T}_1$, and $\hat{T}_2$ with infinite boundary slope. 

    Let $T_i\times[0,1]$ denote the toric annulus in $\Sigma_1\times S^1$ with $T_i\times\{0\}=T_i$ and $T_i\times\{1\}=\hat{T}_i$. Note that for $i=1,2$, $T_i\times[0,1]$ is a basic slice. 
    It now follows from Corollary \ref{prop:rigid} that $T_0,T_i$ must be a nonrigid pair for $i=1,2$; hence $T_1,T_2$ must be a rigid pair.
    Thus the signs of the basic slices $T_1\times[0,1]$ and $T_2\times[0,1]$ must be the same and different than the sign of the outermost basic slice of $T_0\times[0,1]$. 

    Now applying Lemma 4.13 in \cite{GS} to $T_1$ and $T_2$ allows us to find a convex torus $\tilde{T}_0$ parallel to $\hat{T}_0$ with slope $1$. Hence there exists toric annulus $T_0\times[0,2]$ such that $T_0\times\{0\}=T_0$, $T_0\times\{1\}=\hat{T}_0$, and $T_2\times\{2\}=\tilde{T}_0$. Hence $\Sigma_1\times S^1$ is not minimally twisting.
\end{proof}
 
We are now ready to prove Theorem \ref{thm:mainex}; we break it into the following two Propositions.

\begin{prop} $Y$ admits at most $$(a_1-1)\cdots(a_n-1)\prod_{i=1}^4(b_1^i-1)\cdots(b^i_{k_i}-1)$$
minimally twisting tight contact structures.
    \label{prop:mintwistex}
\end{prop}

\begin{proof}
By Lemma \ref{lem:twisting}, there are no vertical Legendrian curves of twisting number 0 in $\Sigma_i\times S^1$ for $i=1,2$. Hence by \cite{H2}, there are unique tight contact structures on $\Sigma_1\times S^1$ and $\Sigma_2'\times S^1$.
Moreover, since 
$$\frac{p-q}{p'-q'}=[a_1,\ldots,a_{n-1},a_n-1],$$
$T^2\times[0,1]$ admits exactly $(a_1-1)\cdots(a_n-1)$ tight contact structures. 
Since $s(\Gamma_{T_i})=-1$, we have that $s(\Gamma_{\partial V_i})=-\frac{x_i-y_i}{x_i'-y_i'}$. Again by Honda, we have that $V_i$ admits exactly $(b_1^i-1)\cdots(b_{k_i}^i-1)$ tight contact structures. 

Gluing the pieces together yields at most $(a_1-1)\cdots(a_n-1)\prod_{i=1}^4(b_1^i-1)\cdots(b^i_{k_i}-1)$ tight contact structures on $Y$.
\end{proof}

Let $T$ denote an incompressible torus in $Y$, which is contained in the maximal chain of $Y$. Suppose $(Y,\xi)$ is a tight contact structure with $\frac{m}{2}-$twisting along $T$, where $m>0$. It follows that $\Sigma_i\times S^1$ for some $i$ contains a vertical Legendrian with twisting number 0. Following the proof of Lemma \ref{lem:twisting}, and recutting $Y$, we may assume that $\Sigma_i\times S^1$ will have boundary slopes $\infty$.

\begin{prop} Let $m\in\ZZ_{\ge1}$. Then $Y$ admits at most $$2\prod_{i=1}^4(b_2^i-1)\cdots(b^i_{k_i}-1)$$
tight contact structures with $\frac{m}{2}$-twisting.
    \label{prop:twistingex}
\end{prop}

\begin{proof}
By construction, $s(\Gamma_{T_i})=-1$ for all $i$ and, up to recutting, we may assume that the $T^2\times I$ contains the $\frac{m}{2}-$twisting. Since $T^2\times I\subset\Sigma_2\times S^1$, the latter contains a vertical Legendrian. Hence we may recut $Y$ so that the boundary slopes of $T_0^2$, $T_3$, $T_4$ are infinite.
By Corollary \ref{prop:rigid}, the pair $T_3,T_4$ must be rigid. Since the first continued fraction blocks of $V_3$ and $V_4$ must have the same signs, we have two possible contact structures on those continued fraction blocks. The remainder of $V_i$ admits $(b_2^i-1)\cdots(b_{k_i}^i-1)$ tight contact structures for $i=3,4$. Since the outermost basic slice of $T^2\times I$ must have the opposite sign of the outermost basic slices of $V_1$ and $V_2$, there is a unique contact structure on $T^2\times I$. Hence the portion of $Y$ consisting of $\Sigma_2\times S^1$ and the solid tori $V_3$ and $V_4$ admits at most $2\prod_{i=3,4}(b_2^i-1)\cdots(b_{k_i}^i-1)$ tight contact structures. 

Now, recut $Y$ so that $T^2\times I$ is glued to $\Sigma_1\times I$. As above, $T_1,T_2$ must be rigid. Since the first continued fraction blocks of $V_1$ and $V_2$ must have the same signs, we have two possible contact structures on those continued fraction blocks.
However, the signs of the outermost basic slices of $V_3$ and $V_4$ determines the signs of the basic slices of $T^2\times I$, which in turn must determine the signs of the outermost basic slices of $V_1$ and $V_2$. Hence there is only one possible contact structure on the these continued fraction blocks. Now, as above, the remainder of $V_i$ admits $(b_2^i-1)\cdots(b_{k_i}^i-1)$ tight contact structures for $i=1,2$. Hence the portion of $Y$ consisting of $\Sigma_1\times S^1$, $T^2\times I$ and the solid tori $V_1$ and $V_2$ admits at most $\prod_{i=1,2}(b_2^i-1)\cdots(b_{k_i}^i-1)$ tight contact structures.
It now follows that $Y$ admits at most $2\prod_{i=1}^4(b_2^i-1)\cdots(b_{k_i}^i-1)$ right contact structures. \end{proof}

\section{Stein Diagrams: Proof of Theorem \ref{thm:mainminimallytwisting}}\label{sec:stein}
In this section we describe an algorithm to draw Stein diagrams for plumbed 3-manifolds whose associated graphs have no bad vertices and whose vertices are all at most trivalent.
It is a general fact from graph theory that such a graph is planar unless it is the complete bipartite graph $\mathcal{K}_{3,3}$. 

Over the next three subsections, we will consider the following three cases:
\begin{itemize}
    \item $\Gamma$ is planar and 2-connected;
    \item $\Gamma=\mathcal{K}_{3,3}$; and
    \item the general case.
\end{itemize}

\subsection{$\Gamma$ is planar and 2-connected}
Assume that $\Gamma$ is planar and 2-connected (i.e. $\Gamma$ cannot be made disjoint by removing a single vertex). Such graphs have cycles arranged in a cluster; see the top left of Figure \ref{fig:cluster} for an example (ignoring the red curve). We would like to draw $\Gamma$ so that it is in the following \textit{wrapped-up form}:
\begin{itemize}
    \item vertices are arranged in (horizontal) rows;
    \item the bottom row of vertices lie on a linear subgraph of $\Gamma$ and the endpoints of this subgraph are connected by an curved edge $\gamma$ below the bottom row, giving a cycle $c$ that does not enclose any portion of the graph;
    \item the first and last vertex of each row are each incident to curved edges that wrap around the edge $\gamma$; 
    \item all edges are either horizontal, vertical, or curved edges that wrap around the edge $\gamma$; and
    \item every cycle of $\Gamma$ encloses the innermost cycle $c$.
    \end{itemize}
See the bottom of Figure \ref{fig:cluster} for an example. 
We place our graph in wrapped-up form, because we can then easily apply the techniques in \cite{GoS} to draw a handlebody diagram of the 4-dimensional plumbing described by the graph. In particular, their will be: a 1-handle for every curved edge; a 2-handle attaching circle for every vertex; and the attaching circles of the 2-handles will link according to the edge structure of the graph (see \cite{GoS} for details).
We now aim to prove that that any planar 2-connected graph with vertices that are at most trivalent can be isotoped (in $S^2$) into wrapped-up form. 
We first need a graph-theoretic result.

\begin{lem} Let $\Gamma$ be a planar, 2-connected graph whose vertices are at most trivalent that is not a loop. Then
up to isotopy in $S^2$, the dual graph $\Gamma^*$ of $\Gamma$ contains a Hamiltonian path (a path traversing every vertex precisely once) ending at $v_\infty$, the vertex of $\Gamma^*$ corresponding to the unbounded region of $\mathbb{R}^2\setminus\Gamma$.
\label{lem:hamiltonian}
\end{lem}

\begin{proof}
We proceed by induction on the number of vertices $n$ of the dual graph $\Gamma^*$.
If $n=2$, then $\Gamma$ is simply a cycle and $\Gamma^*$ clearly contains a Hamiltonian path ending at $v_\infty$.  
Assume that every dual graph with $k$ vertices has a Hamiltonian path ending at $v_\infty$ such that the second-to-last vertex in the path can be chosen to be any vertex of $\Gamma_{k+1}^*$ adjacent to $v_\infty$.
Let $\Gamma_{k+1}$ be a graph (satisfying the hypotheses of the lemma) whose dual graph $\Gamma_{k+1}^*$ has $k+1\ge3$ vertices.
Let $v_z\in\Gamma_{k+1}^*$ be an arbitrary vertex adjacent to $v_\infty$.
Then $v_z$ corresponds to a bounded region $R$ of $\mathbb{R}^2\setminus \Gamma_{k+1}$ and $v_\infty$ corresponds to the unbounded region $R_\infty$ of $\mathbb{R}^2\setminus \Gamma_{k+1}$. Since $k+1\ge3$, it follows that $v_z$ is adjacent to another vertex $v_j\in\Gamma_{k+1}^*$. 
Let $C$ denote the linear subgraph of $\Gamma_{k+1}$ between two trivalent vertices that borders $R$ and $R_\infty$. Let $\Gamma_k$ denote the graph obtained from $\Gamma_{k+1}$ by removing all vertices and edges of $C$, except for the trivalent vertices at the beginning and end of $C$. Then $\Gamma_{k}^*$ is obtained from $\Gamma_{k+1}^*$ by removing $v_z$ along with all of the edges incident to $v_z$ and adding edges incident to $v_\infty$, one of which connects to $v_j$.
Since $\Gamma_k^*$ has length $k$, it
has a Hamiltonian path $a$ whose last two vertices are $v_j$ and $v_\infty$; denote this path by $(v_1,\ldots,v_{k-2},v_j,v_\infty)$. Since $v_j$ and $v_\infty$ are both adjacent to $v_z$ in $\Gamma_{k+1}^*$, the path $a'=(v_1,\ldots,v_{k-2},v_j,v_z,v_\infty)$ is a Hamiltonian path of $\Gamma_k^*$ ending at $v_\infty$. 
\end{proof}

\begin{lem}[Wrapping Algorithm]
Let $\Gamma$ be a 2-connected planar graph in which each vertex has valence at most 3. Then $\Gamma$ can be isotoped in $S^2$ into wrapped-up form.
\label{lem:wrappingalgorithm}
\end{lem}

\begin{proof}
By Lemma \ref{lem:hamiltonian}, the dual graph of $\Gamma$ contains a Hamiltonian path $c$ ending at $v_\infty$, the vertex corresponding to the unbounded region of $\mathbb{R}^2\setminus\Gamma$. Hence there exists a bounded region $R_0$ in $\mathbb{R}^2\setminus \Gamma$ and a path $a$ from $R_0$ to the unbounded region that passes through every region of $\mathbb{R}^2\setminus \Gamma$ precisely once; see the first diagram in Figure \ref{fig:cluster} for an example. As we traverse $a$, it passes through $n$ edges until it reaches the unbounded region; denote these edges by $e_0,\ldots, e_{n-1}$. For $1\le i\le n-1$, let $\gamma_i$ denote the linear subgraph of $\Gamma$ containing $e_i$ and ending in two trivalent vertices of $\Gamma$. Now isotope $\gamma_{n-1}$ in $S^2$ through the point at infinity to the other side of $\Gamma$ (see the second diagram in Figure \ref{fig:cluster}). 
Similarly isotope $\Gamma_{n-2},\ldots, \Gamma_1$ in order (see the third diagram in Figure \ref{fig:cluster}). The ending result can then be isotoped into wrapped-up form by isotoping the graph so that $a$ is a vertical line traveling upward and the vertices are arranged in rows (see the final diagram in Figure \ref{fig:cluster}).
\end{proof}

\begin{figure}
    \begin{tikzpicture}[dot/.style = {circle, fill, minimum size=1pt, inner sep=0pt, outer sep=0pt}]
\tikzstyle{smallnode}=[circle, inner sep=0mm, outer sep=0mm, minimum size=2mm, draw=black, fill=black];

\node[smallnode] (x) at (0,0) {};
\node[smallnode] (y) at (.866,-.5) {};
\node[smallnode] (z) at (0,-1.4) {};
\node[smallnode] (w) at (-.866,-.5) {};
\node[smallnode] (a) at (1.6,.4) {};
\node[smallnode] (b) at (1,1) {};
\node[smallnode] (c) at (0,1) {};
\node[smallnode] (d) at (-1,1) {};
\node[smallnode] (e) at (-1.6,.4) {};
\draw[-] (x) -- (y);
\draw[-] (y) -- (z);
\draw[-] (z) -- (w);
\draw[-] (w) -- (x);

\draw[-] (a) -- (b);
\draw[-] (b) -- (c);
\draw[-] (c) -- (d);
\draw[-] (d) -- (e);

\draw[-] (a) -- (y);
\draw[-] (c) -- (x);
\draw[-] (e) -- (w);

\node (h) at (0, -.5){};
\node (g) at (.5,.5){};
\node (i) at (-1.8,.75){};
\draw[color=red] (h.center) [out=45, in=0] to (g.center);
\draw[color=red] (g.center) [out=170, in=0] to (i.center);
\node[color=red] at (.8, .5){$a$};


\node (arrow1) at (2,0){};
\node (arrow2) at (3.6,0){};
\draw[->] (arrow1) -- (arrow2);

\node[smallnode] (x1) at (4.8,0) {};
\node[smallnode] (y1) at (5.666,-.5) {};
\node[smallnode] (z1) at (4.8,-1.4) {};
\node[smallnode] (w1) at (3.934,-.5) {};
\node[smallnode] (a1) at (6.4,.4) {};
\node[smallnode] (b1) at (5.8,1) {};
\node[smallnode] (c1) at (4.8,1) {};
\node[smallnode] (d1) at (7.3,1.3) {};
\node[smallnode] (e1) at (4.8,-2.5) {};
\draw[-] (x1) -- (y1);
\draw[-] (y1) -- (z1);
\draw[-] (z1) -- (w1);
\draw[-] (w1) -- (x1);

\draw[-] (a1) -- (b1);
\draw[-] (b1) -- (c1);
\draw[-] (c1) -- (d1);
\draw[-] (d1) -- (e1);

\draw[-] (a1) -- (y1);
\draw[-] (c1) -- (x1);
\draw[-] (e1) -- (w1);

\node (h1) at (5, -.5){};
\node (g1) at (5.3,.5){};
\node (i1) at (3.934,.75){};
\draw[color=red] (h1.center) [out=45, in=0] to (g1.center);
\draw[color=red] (g1.center) [out=170, in=0] to (i1.center);
\node[color=red] at (5.6, .5){$a$};


\node (arrow1) at (7.1,0){};
\node (arrow2) at (8.8,0){};
\draw[->] (arrow1) -- (arrow2);

\node[smallnode] (x2) at (10,0) {};
\node[smallnode] (y2) at (10.866,-.5) {};
\node[smallnode] (z2) at (10,-1.4) {};
\node[smallnode] (w2) at (9.134,-.5) {};
\node[smallnode] (a2) at (11.6,.4) {};
\node[smallnode] (b2) at (11,1) {};
\node[smallnode] (c2) at (10,1) {};
\node[smallnode] (d2) at (12.5,1.3) {};
\node[smallnode] (e2) at (10,-2.5) {};
\draw[-] (x2) -- (y2);
\draw[-] (y2) -- (z2);
\draw[-] (z2) -- (w2);
\draw[-] (w2) -- (x2);

\draw[-] (a2) -- (b2);
\draw[-] (b2) -- (c2);
\draw[-] (c2) -- (d2);
\draw[-] (d2) -- (e2);

\draw[-] (a2) -- (y2);
\draw[-] (e2) -- (w2);

\node (p) at (13,1.4) {};
\draw[-] (c2) [out=90, in=90] to (p.center);
\node (q) at (10,-3) {};
\draw[-] (p.center) [out=270, in=0] to (q.center);
\node (r) at (8.7,-.5) {};
\draw[-] (q.center) [out=180, in=270] to (r.center);
\draw[-] (r.center) [out=90, in=90] to (x2);

\node (h) at (10.1, -.5){};
\node (g) at (10.5,.5){};
\node (i) at (9,.75){};
\draw[color=red] (h.center) [out=45, in=0] to (g.center);
\draw[color=red] (g.center) [out=170, in=0] to (i.center);
\node[color=red] at (10.8, .5){$a$};


\node (arrow1) at (10,-3.2){};
\node (arrow2) at (8.2,-4){};
\draw[->] (arrow1) -- (arrow2);

\node[smallnode] (x3) at (3.5,-5) {};
\node[smallnode] (y3) at (4.5,-5) {};
\node[smallnode] (z3) at (5.5,-5) {};
\node[smallnode] (w3) at (6.5,-5) {};
\node[smallnode] (a3) at (5.5,-4) {};
\node[smallnode] (b3) at (6,-4) {};
\node[smallnode] (c3) at (6.5,-4) {};
\node[smallnode] (d3) at (7,-4) {};
\node[smallnode] (e3) at (3.5,-4) {};
\draw[-] (x3) -- (w3);
\draw[-] (a3) -- (d3);

\draw[-] (x3) -- (e3);
\draw[-] (z3) -- (a3);

\node (p0) at (3.5,-3) {};
\draw[-] (y3) [out=90, in=0] to (p0.center);
\node (p1) at (3.5,-6.5) {};
\draw[-] (p0.center) [out=180, in=180] to (p1.center);
\node (p2) at (6.5,-6.5) {};
\draw[-] (p1.center) [out=0, in=180] to (p2.center);
\node (p3) at (8,-4) {};
\draw[-] (p2.center) [out=0, in=270] to (p3.center);
\draw[-] (p3.center) [out=90, in=90] to (c3);

\node (q1) at (3.5,-6) {};
\draw[-] (e3) [out=180, in=180] to (q1.center);
\node (q2) at (6.5,-6) {};
\draw[-] (q1.center) [out=0, in=180] to (q2.center);
\draw[-] (q2.center) [out=0, in=0] to (d3.center);

\node (r1) at (3.5,-5.5) {};
\draw[-] (x3) [out=180, in=180] to (r1.center);
\node (r2) at (6.5,-5.5) {};
\draw[-] (r1.center) [out=0, in=180] to (r2.center);
\draw[-] (r2.center) [out=0, in=0] to (w3.center);

\node (h2) at (5, -5.2){};
\node (g2) at (5,-3.5){};
\draw[color=red] (h2.center) [out=90, in=270] to (g2.center);
\node[color=red] at (5.2, -3.5){$a$};

\end{tikzpicture}
\caption{Redrawing a cluster}\label{fig:cluster}
    \label{fig:dbc}
\end{figure}
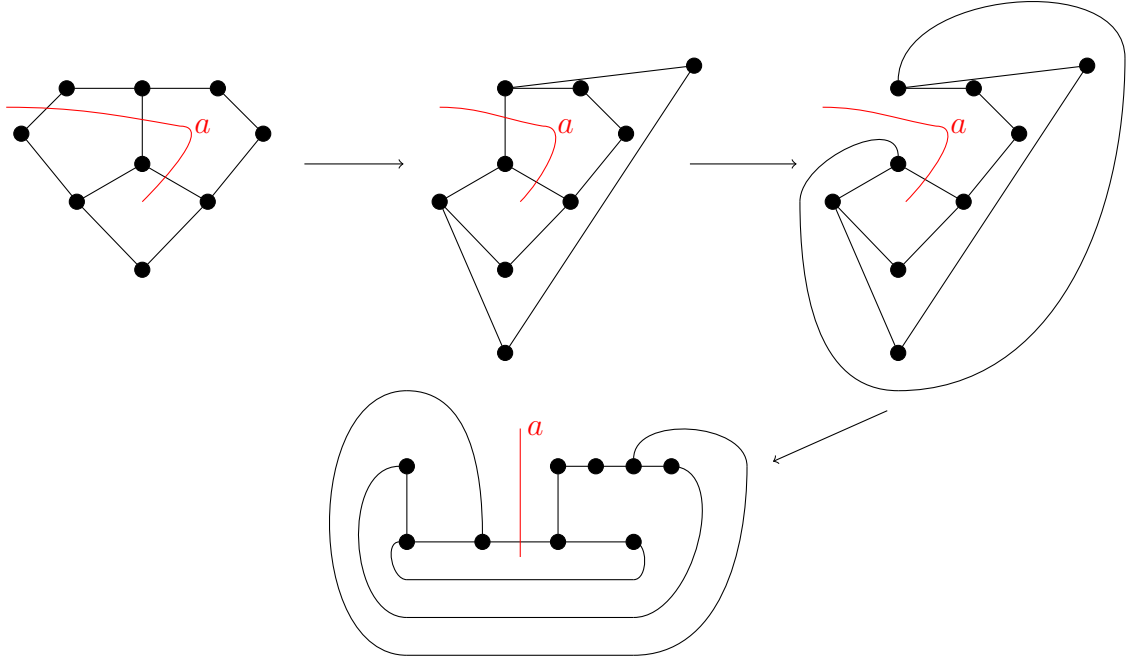

Once $\Gamma$ is in this wrapped-up form, we can begin drawing a handlebody diagram of $P$ whose 1-handles are in a standard position as in \cite{GoS}; recall that there is a 1-handle for every curved edge. Moreover, note that each 2-handle can run through at most two 1-handles, since the vertices in the original graph are all at most trivalent. 
We will describe how to draw this diagram in two steps, which will be useful when arranging the handlebody diagram into a Stein diagram. 

\underline{Smooth Step 1:} Draw a (clockwise oriented) unknot for every vertex in each row and link any two unknots in the same row corresponding to adjacent vertices; the linking should be positive if the corresponding edge is positive and negative if the corresponding edge is negative (see Figures \ref{fig:poslink} and \ref{fig:neglink}). If a vertex is at the end of a row and is incident to a unique curved edge, then pass the corresponding unknot through the corresponding 1-handle and link it with the other vertex incident to the curved edge (see Figure \ref{fig:linkthru1handle}). If a vertex is at the end of a row and is incident to two curved edges, then pass the corresponding unknot through both of the corresponding 1-handles and link it with the other vertices incident to the curved edges (see Figure \ref{fig:linkthrutwo1handles}). For any curved edge not incident to the first or last vertices of a row, draw a ``half" unknot protruding from the 1-handle (see Figure \ref{fig:protrude}).
For a full example of Step 1, see Figure \ref{fig:handlebody1}. 

\begin{figure}
\centering
\begin{subfigure}{.45\textwidth}
\captionsetup{width=.7\textwidth}
         \centering
         \includegraphics[scale=.6]{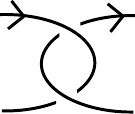}
         \caption{Linking corresponding to a positive edge}
         \label{fig:poslink}
\end{subfigure}
\begin{subfigure}{.45\textwidth}
\captionsetup{width=.7\textwidth}
         \centering
         \includegraphics[scale=.6]{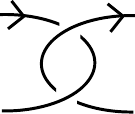}
         \caption{Linking corresponding to a negative edge}
         \label{fig:neglink}
\end{subfigure}
\begin{subfigure}{.45\textwidth}
\captionsetup{width=.7\textwidth}
         \centering
         \vspace{1cm}
         \includegraphics[scale=.45]{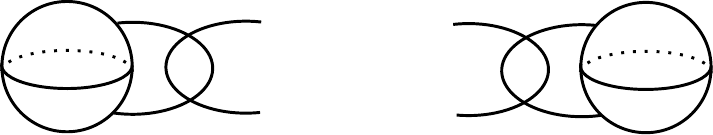}
         \caption{Passing unknot through a 1-handle (note that the linking information is left out)}
         \label{fig:linkthru1handle}
\end{subfigure}
\begin{subfigure}{.45\textwidth}
\captionsetup{width=.7\textwidth}
         \centering
         \vspace{1cm}
         \includegraphics[scale=.45]{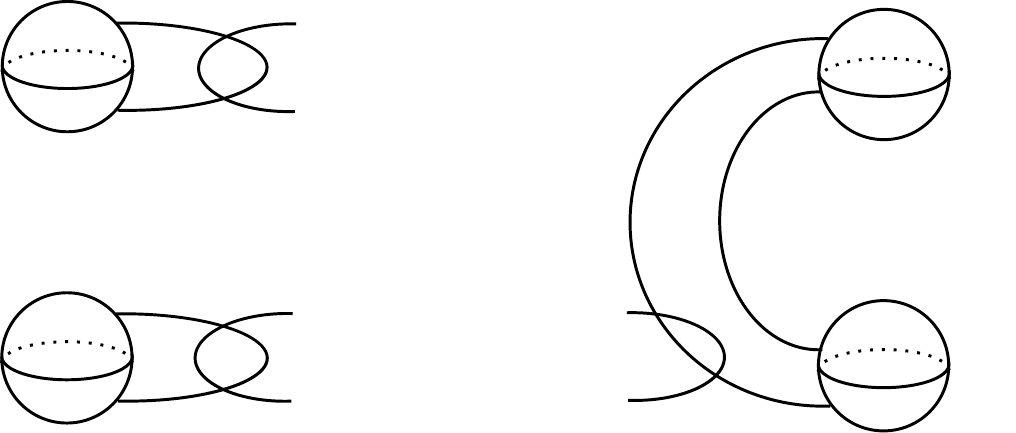}
         \caption{Passing unknot through two 1-handles (note that the linking information is left out)}
         \label{fig:linkthrutwo1handles}
\end{subfigure}
\begin{subfigure}{.45\textwidth}
\captionsetup{width=.7\textwidth}
         \centering
         \vspace{1cm}
         \includegraphics[scale=.45]{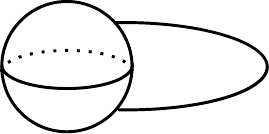}
         \caption{A ``half unknot" protruding through a 1-handle}
         \label{fig:protrude}
\end{subfigure}
\begin{subfigure}{.45\textwidth}
\captionsetup{width=.7\textwidth}
         \centering
         \vspace{1cm}
         \includegraphics[scale=.45]{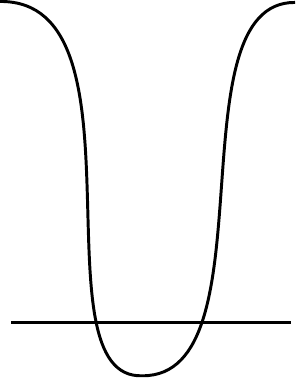}
         \caption{Linking two unknots according to a vertical edge}
         \label{fig:verticaledge}
\end{subfigure}
\begin{subfigure}{.45\textwidth}
\captionsetup{width=.7\textwidth}
         \centering
         \vspace{1cm}
         \includegraphics[scale=.45]{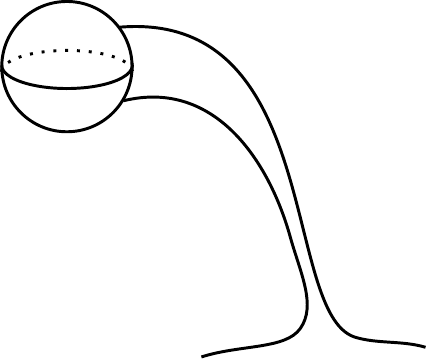}
         \caption{Adding a bank according to a vertical curved edge}
         \label{fig:verticalcurvededge}
\end{subfigure}
\caption{Drawing 2-handles locally}\label{fig:local}
\end{figure}

\underline{Smooth Step 2:} For each vertical edge between two vertices in the plumbing graph, add a band to the top unknot $K_1$ and link this band with the lower unknot $K_2$ (see Figure \ref{fig:verticaledge}). For each curved edge incident to a vertex not at the beginning or end of a row, similarly add a band between the relevant unknot protruding from a 1-handle and the unknot below corresponding to the vertex (see \ref{fig:verticalcurvededge}). See Figure \ref{fig:handlebody2} for a full example.\\

\begin{figure}
\centering
\begin{subfigure}{\textwidth}
\centering
\captionsetup{width=0.7\textwidth}
    \begin{tikzpicture}[dot/.style = {circle, fill, minimum size=1pt, inner sep=0pt, outer sep=0pt}]
    \tikzstyle{smallnode}=[circle, inner sep=0mm, outer sep=0mm, minimum size=2mm, draw=black, fill=black];
    \node[smallnode] (x3) at (3.5,-5) {};
    \node[smallnode] (y3) at (4.5,-5) {};
    \node[smallnode] (z3) at (5.5,-5) {};
    \node[smallnode] (w3) at (6.5,-5) {};
    \node[smallnode] (a3) at (5.5,-4) {};
    \node[smallnode] (b3) at (6,-4) {};
    \node[smallnode] (c3) at (6.5,-4) {};
    \node[smallnode] (d3) at (7,-4) {};
    \node[smallnode] (e3) at (3.5,-4) {};
    \draw[-] (x3) -- (w3);
    \draw[-] (a3) -- (d3);

    \draw[-] (x3) -- (e3);
    \draw[-] (z3) -- (a3);

    \node (p0) at (3.5,-3) {};
    \draw[-] (y3) [out=90, in=0] to (p0.center);
    \node (p1) at (3.5,-6.5) {};
    \draw[-] (p0.center) [out=180, in=180] to (p1.center);
    \node (p2) at (6.5,-6.5) {};
    \draw[-] (p1.center) [out=0, in=180] to (p2.center);
    \node (p3) at (8,-4) {};
    \draw[-] (p2.center) [out=0, in=270] to (p3.center);
    \draw[-] (p3.center) [out=90, in=90] to (c3);

    \node (q1) at (3.5,-6) {};
    \draw[-] (e3) [out=180, in=180] to (q1.center);
    \node (q2) at (6.5,-6) {};
    \draw[-] (q1.center) [out=0, in=180] to (q2.center);
    \draw[-] (q2.center) [out=0, in=0] to (d3.center);

    \node (r1) at (3.5,-5.5) {};
    \draw[-] (x3) [out=180, in=180] to (r1.center);
    \node (r2) at (6.5,-5.5) {};
    \draw[-] (r1.center) [out=0, in=180] to (r2.center);
    \draw[-] (r2.center) [out=0, in=0] to (w3.center);

    \node at (5, -4.8){$-$};
    \node at (7, -3.4){$-$};
    \node at (5.75, -3.8){$-$};
    \node at (5.75, -4.5){$-$};
    \end{tikzpicture}
    \caption{The original plumbing in wrapped-up form}
    \label{fig:originalP}
\end{subfigure}
\begin{subfigure}{\textwidth}
\captionsetup{width=0.7\textwidth}
    \vspace{1cm}
         \centering
         \includegraphics[scale=.5]{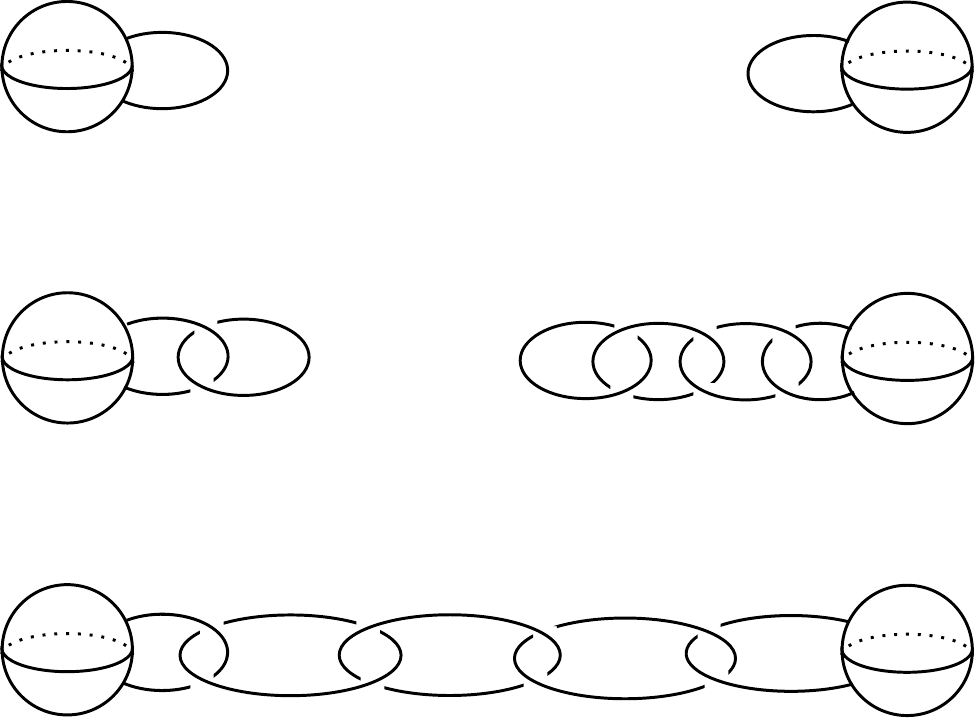}
         \caption{Step 1: draw a horizontal unknot in each vertex position}
         \label{fig:handlebody1}
\end{subfigure}
\begin{subfigure}{\textwidth}
\captionsetup{width=0.7\textwidth}
         \centering
         \vspace{1cm}
         \includegraphics[scale=.5]{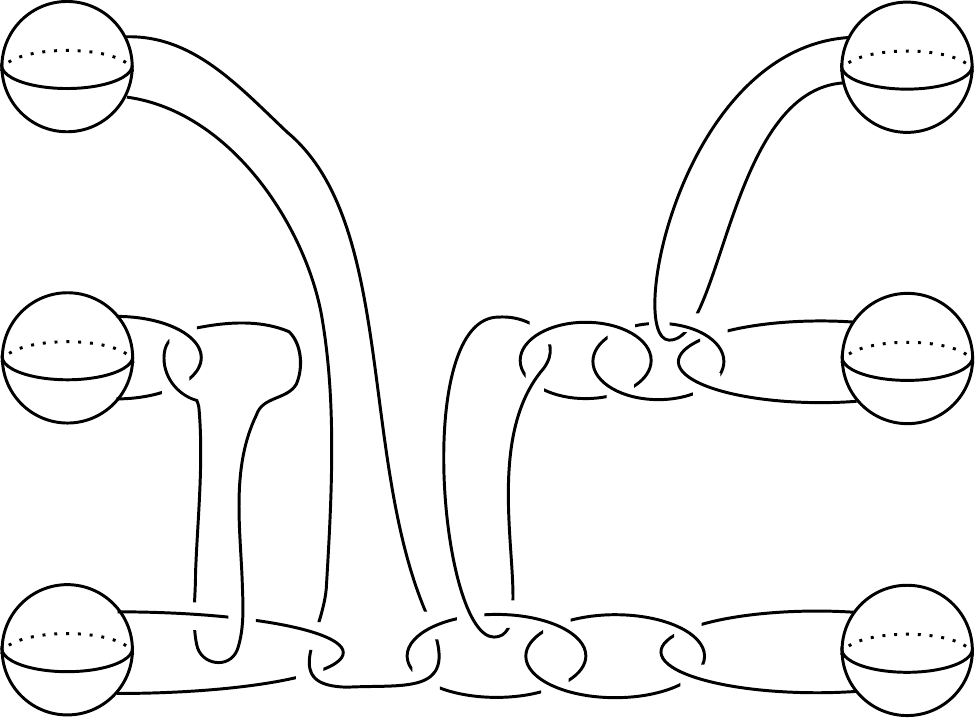}
         \caption{Step 2: For each vertical linking, add a band}
         \label{fig:handlebody2}
\end{subfigure}
\caption{Drawing a handlebody diagram corresponding to the clustered plumbing in Figure \ref{fig:cluster}}
\end{figure}

We now show that we can draw each unknot in this handlebody diagram as a Legendrian unknot with $tb=-1$. We do this again in two steps.

\underline{Legendrian Step 1}: Start with the diagram in Smooth Step 1 above and draw every unknot not passing through two 1-handles as the standard Legendrian unknot with $tb=-1$ and link two adjacent unknots as in either Figure \ref{fig:pos} or \ref{fig:neg}. 
For each unknot passing through two 1-handles, we can slide the unknot over one of the 1-handles and arrange the diagram as in Figure \ref{fig:1-handle2} and link the unknot to any adjacent unknot in the same row according to either Figure \ref{fig:pos} or \ref{fig:neg}. 
See Figure \ref{fig:stein0} for the continued example.

\underline{Legendrian Step 2}: Suppose $K_1$ and $K_2$ are unknots corresponding to adjacent vertices such that $K_1$ is higher in the diagram than $K_2$. If the edge between the vertices is negative, then arrange the linking as in Figure \ref{fig:neg2}; if the edge is positive, then after sliding $K_1$ over every 1-handle below $K_1$ in the diagram, we can arrange the linking as in Figure \ref{fig:pos2} (or its reflection).
Next suppose an unknot $K$ runs through a single 1-handle and then travels vertically downward in the diagram (i.e. the vertex corresponding to $K$ is not at the beginning or end of its row and it is incident to a curved edge). In Smooth Step 1, we added a band between the unknot protruding from the 1-handle and $K$; here, we can slide the band over the 1-handle, arranging $K$ near the 1-handle as in Figure \ref{fig:1handle} (or its reflection). \\

\begin{figure}
\centering
\begin{subfigure}{.45\textwidth}
\captionsetup{width=0.7\textwidth}
         \centering
         \includegraphics[scale=.5]{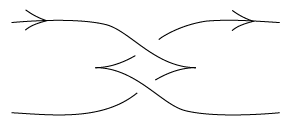}
         \caption{Horizontal linking corresponding to positive edges}
         \label{fig:pos}
\end{subfigure}
\begin{subfigure}{0.45\textwidth}
\captionsetup{width=0.7\textwidth}
         \centering
         \includegraphics[scale=.5]{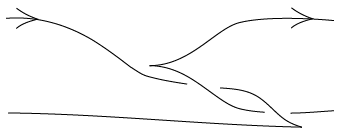}
         \caption{Horizontal linking corresponding to negative edges}
         \label{fig:neg}
\end{subfigure}
\begin{subfigure}{.45\textwidth}
\captionsetup{width=0.7\textwidth}
\vspace{.5cm}
         \centering
         \includegraphics[scale=.5]{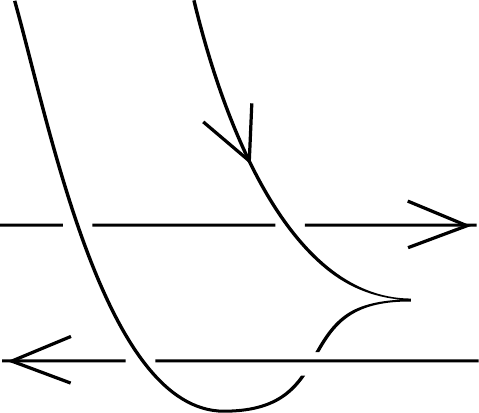}
         \caption{Vertical linking corresponding to positive edges}
         \label{fig:pos2}
\end{subfigure}
\begin{subfigure}{0.45\textwidth}
\captionsetup{width=0.7\textwidth}
\vspace{.5cm}
         \centering
         \includegraphics[scale=.5]{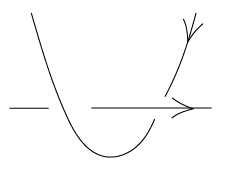}
         \caption{Vertical linking corresponding to negative edges}
         \label{fig:neg2}
\end{subfigure}
\begin{subfigure}{.45\textwidth}
\captionsetup{width=0.7\textwidth}
\vspace{.5cm}
         \centering
         \includegraphics[scale=.5]{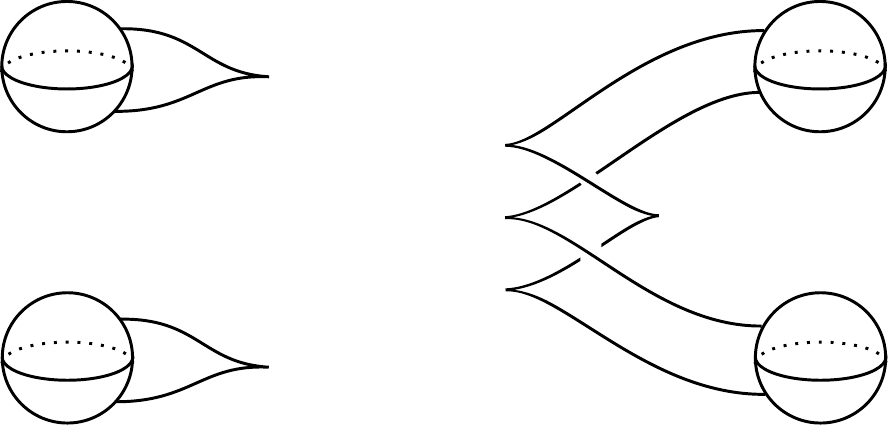}
         \caption{A Legendrian unknot $K$ running through two 1-handles}
         \label{fig:1-handle2}
\end{subfigure}
\begin{subfigure}{.45\textwidth}
\captionsetup{width=0.7\textwidth}
\vspace{.5cm}
         \centering
         \includegraphics[scale=.5]{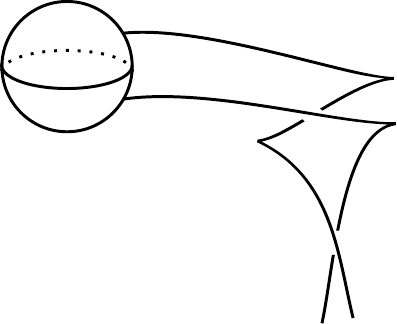}
         \caption{A Legendrian unknot $K$ running a 1-handle}
         \label{fig:1handle}
\end{subfigure}
\label{fig:linking}
\caption{Linking of Legendrian unknots}
\end{figure}

\begin{figure}
\begin{subfigure}{\textwidth}
\captionsetup{width=0.7\textwidth}
         \centering
         \includegraphics[scale=.5]{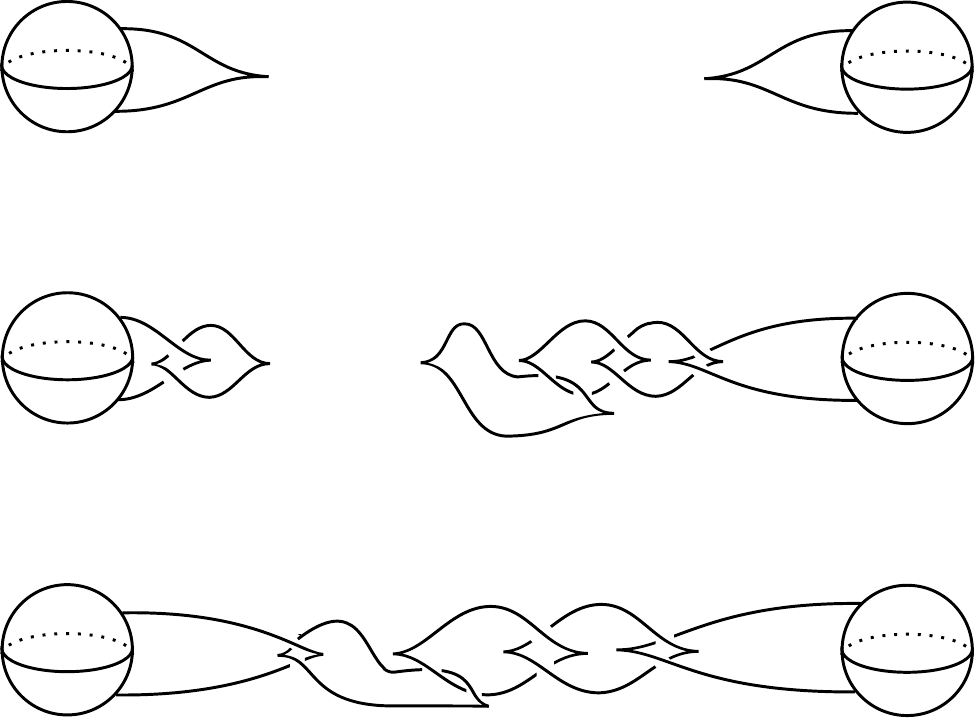}
         \caption{Step 1 of making unknots Legendrian}
         \label{fig:stein0}
\end{subfigure}
\begin{subfigure}{\textwidth}
\captionsetup{width=0.7\textwidth}
\vspace{1cm}
         \centering
         \includegraphics[scale=.5]{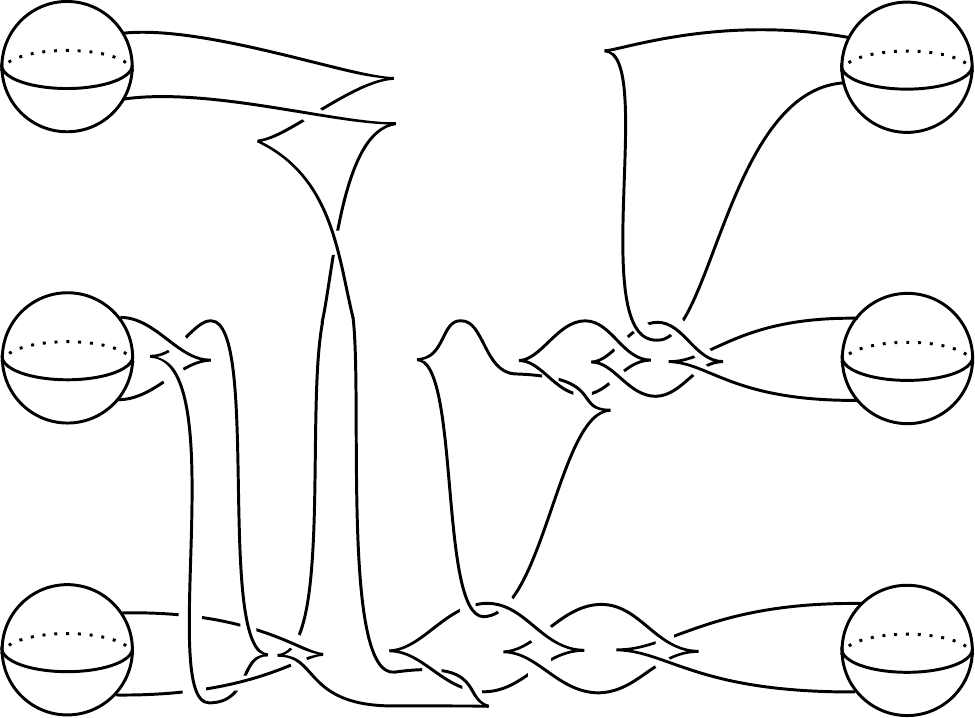}
         \caption{Step 2 of making unknots Legendrian}
         \label{fig:stein1}
\end{subfigure}
\caption{Legendrian handlebody diagram for $P$ corresponding to the graph given in Figure \ref{fig:cluster}. Each unknot has $tb=-1$.}\label{fig:stein}
\end{figure}

The following result follows from the above algorithm.

\begin{prop} If $\Gamma$ is a 2-connected planar graph that are at most trivalent, then the associated plumbed 3-manifold has Legendrian handlebody diagram whose 2-handles are Legendrian unknots with $tb=-1$
    \label{prop:cluster}
\end{prop}

\subsection{$\Gamma$ is nonplanar}
We now consider the graph $\mathcal{K}_{3,3}$, shown in Figure \ref{fig:k33a}, which is the only nonplanar graph having vertices that are at most trivalent.

\begin{prop}
If $\Gamma=\mathcal{K}_{3,3}$, then the associated plumbed 3-manifold has Legendrian handlebody diagram whose 2-handles are Legendrian unknots with $tb=-1$.
\label{prop:k33}
\end{prop}

\begin{proof}
Consider the alternate wrapped-up diagram $\mathcal{K}_{3,3}$ in Figure \ref{fig:k33b}; note that we are leaving out the signs of the edges for convenience. This diagram can be seen to be $\mathcal{K}_{3,3}$ by partitioning the vertices into two sets in which the first set contains the first, third, and fifth vertices (counting from the top left vertex and moving down and to the right). 
Note that since $\mathcal{K}_{3,3}$ is not planar, the pair of intersecting edges cannot be removed.
However, since our diagram is wrapped up as in the planar case, we can follow the algorithm above to draw a Legendrian handlebody diagram in which each unknot has $tb=-1$. First, the smooth handlebody diagram is given in Figure \ref{fig:k33smooth}; the linking information is left out since this depends on the signs of the edges of $\mathcal{K}_{3,3}$. Notice that the top left 2-handle passes over a 2-handle below it (which corresponds to the intersecting edges of $\mathcal{K}_{3,3}$). This diagram is isotopic to the diagram in which the top left 2-handle passes below the other 2-handle; this can been seen by simply sliding the top left 2-handle over the second (from the top) 1-handle. Now if the signs of the edges of $\mathcal{K}_{3,3}$ are specified, then we can apply the algorithm used in the planar case to obtain a Legendrian diagram in which each unknot has $tb=-1$. See Figure \ref{fig:k33stein} for an example of one such diagram (in which the associated graph has exactly two negative edges, which are the outermost and second outermost edges of the wrapped-up graph).
\end{proof}

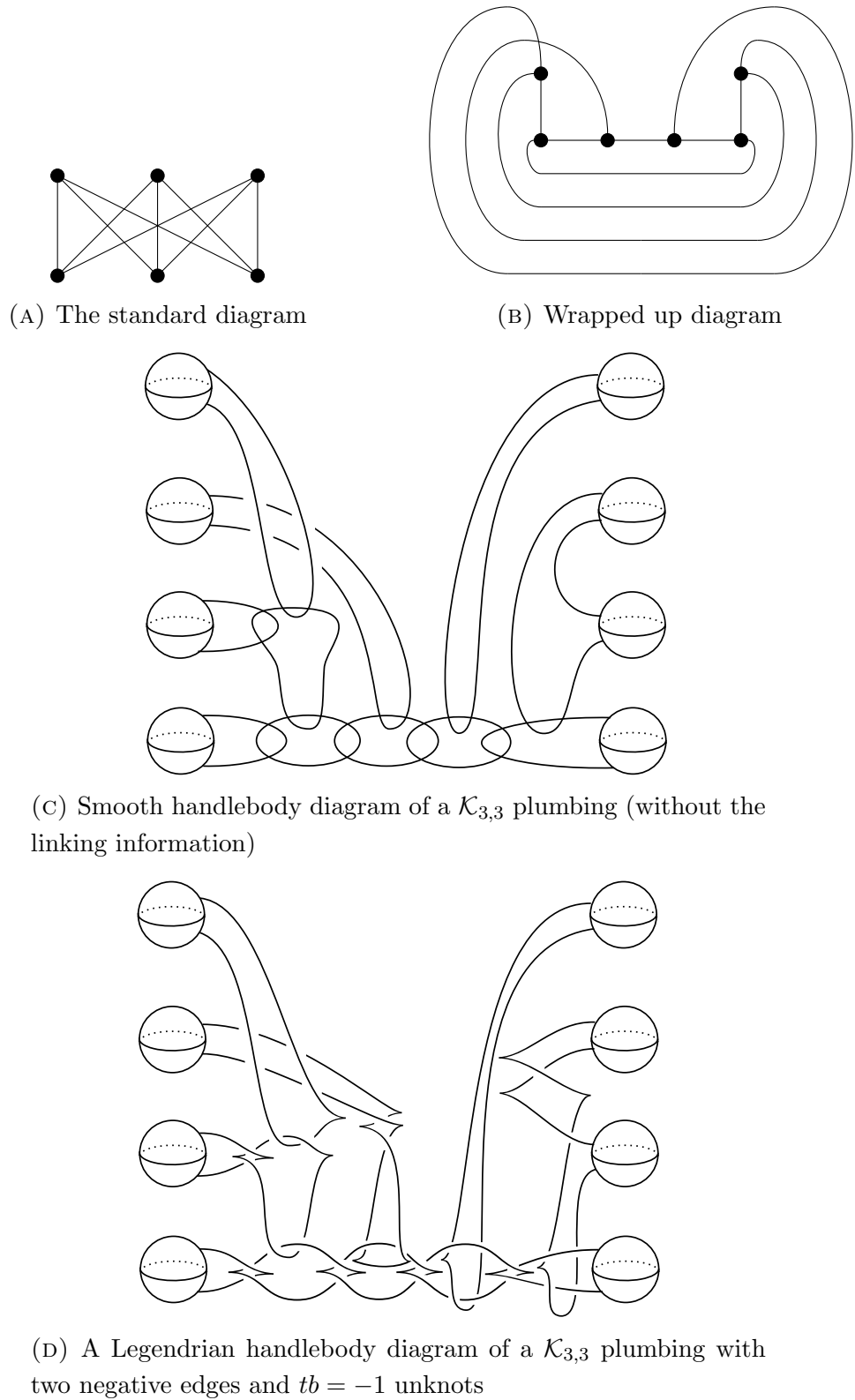
\begin{figure}
    \centering
\begin{subfigure}{.45\textwidth}
    \captionsetup{width=0.7\textwidth}
    \centering
    \begin{tikzpicture}[dot/.style = {circle, fill, minimum size=1pt, inner sep=0pt, outer sep=0pt}]
\tikzstyle{smallnode}=[circle, inner sep=0mm, outer sep=0mm, minimum size=2mm, draw=black, fill=black];
\node[smallnode] (a) at (0,0) {};
\node[smallnode] (b) at (1.5,0) {};
\node[smallnode] (c) at (3,0) {};

\node[smallnode] (d) at (0,-1.5) {};
\node[smallnode] (e) at (1.5,-1.5) {};
\node[smallnode] (f) at (3, -1.5) {};

\draw[-] (a) -- (d);
\draw[-] (a) -- (e);
\draw[-] (a) -- (f);
\draw[-] (b) -- (d);
\draw[-] (b) -- (e);
\draw[-] (b) -- (f);
\draw[-] (c) -- (d);
\draw[-] (c) -- (e);
\draw[-] (c) -- (f);

\end{tikzpicture}
\caption{The standard diagram}\label{fig:k33a}
\end{subfigure}
   \begin{subfigure}{.45\textwidth}
    \captionsetup{width=0.7\textwidth}
    \centering
    \begin{tikzpicture}[dot/.style = {circle, fill, minimum size=1pt, inner sep=0pt, outer sep=0pt}]
\tikzstyle{smallnode}=[circle, inner sep=0mm, outer sep=0mm, minimum size=2mm, draw=black, fill=black];
\node[smallnode] (d) at (0,-1) {};
\node[smallnode] (a) at (1,-1) {};
\node[smallnode] (e) at (2,-1) {};
\node[smallnode] (c) at (3,-1) {};
\node[smallnode] (b) at (0,0) {};
\node[smallnode] (f) at (3, 0) {};

\node (p1) at (-.25,.5){};
\node (p2) at (-.25,-2.5){};
\node (p3) at (1.5,-2.5){};
\node (p4) at (3.25,-2.5){};
\node (p5) at (3.25,.5){};

\node (q1) at (-.5,1){};
\node (q2) at (-.5,-3){};
\node (q3) at (1.5,-3){};
\node (q4) at (3.5,-3){};
\node (q5) at (3.5,1){};

\node (r1) at (0,-2){};
\node (r2) at (3,-2){};

\node (s1) at (0,-1.5){};
\node (s2) at (3,-1.5){};

\draw[-] (a) -- (d);
\draw[-] (a) -- (e);
\draw[-] (c) -- (e);
\draw[-] (c) -- (f);
\draw[-] (b) -- (d);

\draw[-] (a) [out=90, in=0] to (p1.center);
\draw[-] (p1.center) [out=180, in=180] to (p2.center);
\draw[-] (p2.center) [out=0, in=180] to (p3.center);
\draw[-] (p3.center) [out=0, in=180] to (p4.center);
\draw[-] (p4.center) [out=0, in=0] to (p5.center);
\draw[-] (p5.center) [out=180, in=90] to (f);

\draw[-] (b) [out=90, in=0] to (q1.center);
\draw[-] (q1.center) [out=180, in=180] to (q2.center);
\draw[-] (q2.center) [out=0, in=180] to (q3.center);
\draw[-] (q3.center) [out=0, in=180] to (q4.center);
\draw[-] (q4.center) [out=0, in=0] to (q5.center);
\draw[-] (q5.center) [out=180, in=90] to (e);

\draw[-] (b) [out=180, in=180] to (r1.center);
\draw[-] (r1.center) [out=0, in=180] to (r2.center);
\draw[-] (r2.center) [out=0, in=0] to (f);

\draw[-] (d) [out=180, in=180] to (s1.center);
\draw[-] (s1.center) [out=0, in=180] to (s2.center);
\draw[-] (s2.center) [out=0, in=0] to (c);
\end{tikzpicture}
\caption{Wrapped up diagram}\label{fig:k33b}
\end{subfigure}

\begin{subfigure}{\textwidth}
\captionsetup{width=0.7\textwidth}
\vspace{.25cm}
         \centering
         \includegraphics[scale=.45]{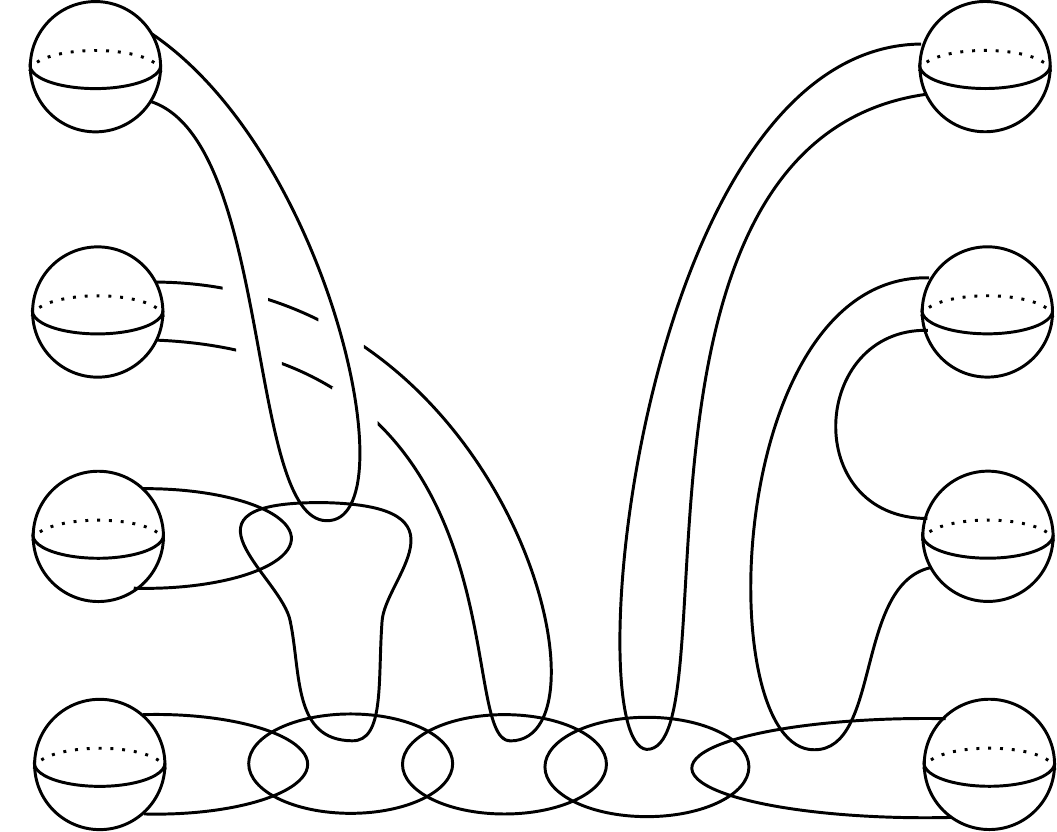}
         \caption{Smooth handlebody diagram of a $\mathcal{K}_{3,3}$ plumbing (without the linking information)}
         \label{fig:k33smooth}
\end{subfigure}
\begin{subfigure}{\textwidth}
\captionsetup{width=0.7\textwidth}
\vspace{.25cm}
         \centering
         \includegraphics[scale=.45]{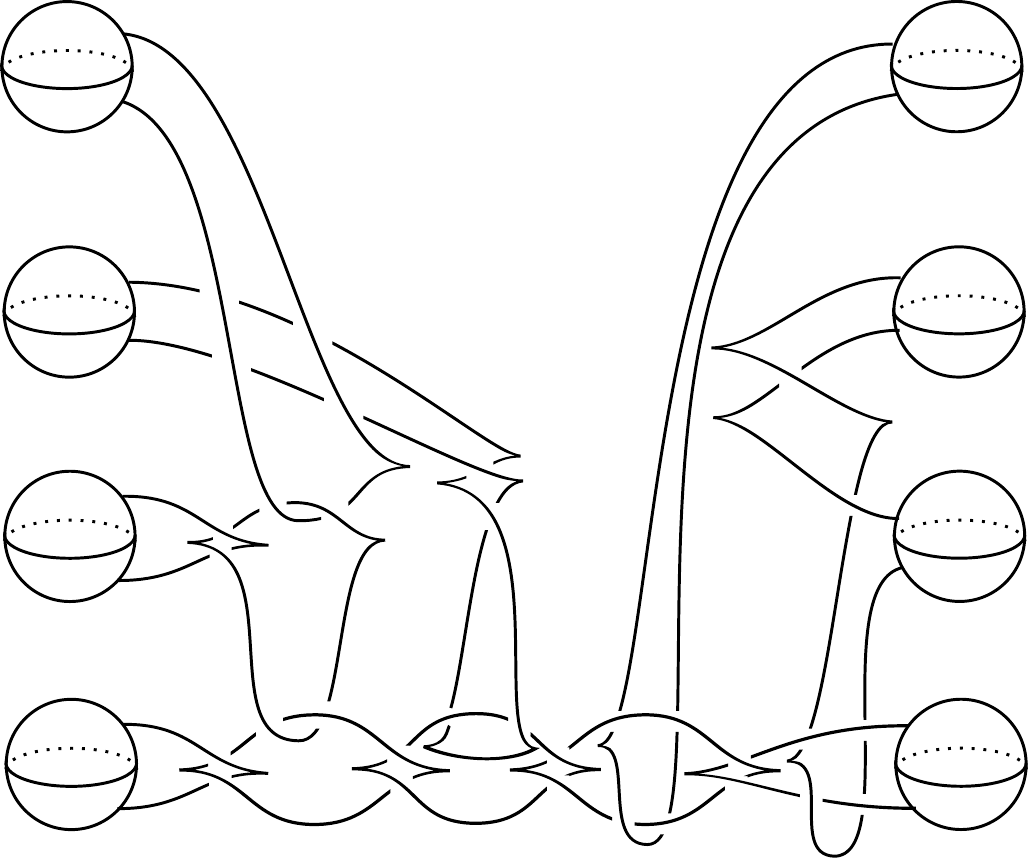}
         \caption{A Legendrian handlebody diagram of a $\mathcal{K}_{3,3}$ plumbing with two negative edges and $tb=-1$ unknots}
         \label{fig:k33stein}
\end{subfigure}

    \caption{The $\mathcal{K}_{3,3}$ plumbing}
    \label{fig:k33}
\end{figure}

\begin{remk}
    The case of $\mathcal{K}_{3,3}$ indicates that the above algorithm for drawing Legendrian diagrams with $tb=-1$ becomes more difficult in the nonplanar case. Hence for more general nonplanar plumbings of higher valence, the algorithm requires some modification.
\end{remk}

\subsection{General Case: Proof of Theorem \ref{thm:mainminimallytwisting}}

Finally, let $\Gamma$ be any plumbing graph with no bad vertices whose vertices are at most trivalent. Let its vertices have weights $-a_1,\ldots,-a_n$. Then it is either $\mathcal{K}_{3,3}$ or it is made of clusters that are connected by tree subgraphs and each cluster might have trees emanating from them. We can thus apply the above algorithm to each cluster, draw each handlebody diagram, stack them vertically, connect diagrams with the appropriate 2-handles, and add any 2-handles associated with any trees emanating from the clusters. We can then draw each unknot as standard Legendrian unknots with $tb=-1$. 
Now for $1\le i\le n$, we stabilize the $a_i-$framed unknot $K_i$ in the diagram $a_i-2$ times so that $tb(K_i)=-(a_i-1)$. Since there are $a_i-1$ ways to stabilize $K_i$, by \cite{LM}, the 4-dimensional plumbing admits $(a_1-1)\ldots(a_n-1)$ distinct Stein structures, which induce distinct Stein fillable contact structures on the boundary plumbed 3-manifold; Theorem 2 follows.

\section{Plumbed 3-manifolds with higher valence vertices}\label{nfibers}
In this section, we consider tight structures on plumbed 3-manifolds with no bad vertices containing vertices with valence $\geq 3.$ What sets these plumbed 3-manifolds apart from plumbed 3-manifolds that have vertices that are at most trivalent is the presence of nontrivially intersecting nonisotopic incompressible tori. To see this, suppose that a plumbed 3-manifold $M$ has a vertex $v$ of valence $n$. Cut $M$ along each of the tori corresponding to the $n$ edges emanating from $v$. This cuts $M$ into two pieces, one of which is a copy of $\Sigma_0^n\times S^1$, where $\Sigma_0^n$ is the $n-$punctured 2-sphere. Notice that any torus in $\Sigma_0^n\times S^1$ of the form $S^1\times S^1\subset \Sigma^n_0\times S^1$, where $S^1\times\{pt\}\subset \Sigma^n_0\times\{pt\}$ is not boundary parallel is an incompressible torus in $M$ (and a torus such that $S^1\times\{pt\}$ is boundary parallel is incompressible if and only if the corresponding boundary component is not the boundary of a solid torus in $M$). Figure \ref{fig:IT} displays four such incompressible tori in $\Sigma_0^5\times S^1$, some of which are disjoint and some of which intersect.
In these plumbed 3-manifolds, Giroux torsion can be added in a neighborhood of any incompressible torus; however, if two such tori intersect nontrivally, then Giroux torsion cannot be added to both simultaneously. On the other hand, if two such tori are disjoint, then Giroux torsion can be added to one or both. The simplest such plumbed 3-manifolds are the so-called called start-shaped plumbings (see Figure \ref{fig:starshaped}); these are Seifert fiber spaces with $n$ singluar fibers.
The first author considered the case when $n=4$ in \cite{Sha}.

The techniques used in \cite{Sha} and in Sections \ref{sec:example} and \ref{sec:stein} above can be used to classify minimally twisting tight structures on plumbed 3-manifolds with no bad vertices.
One can also use the convex surface theory techniques and the rigidity obstruction used in Section \ref{sec:example} to obtain upper bounds on the tight contact structures with Giroux torsion on an ad hoc basis.  

\begin{figure}
\centering
\includegraphics[scale=0.3]{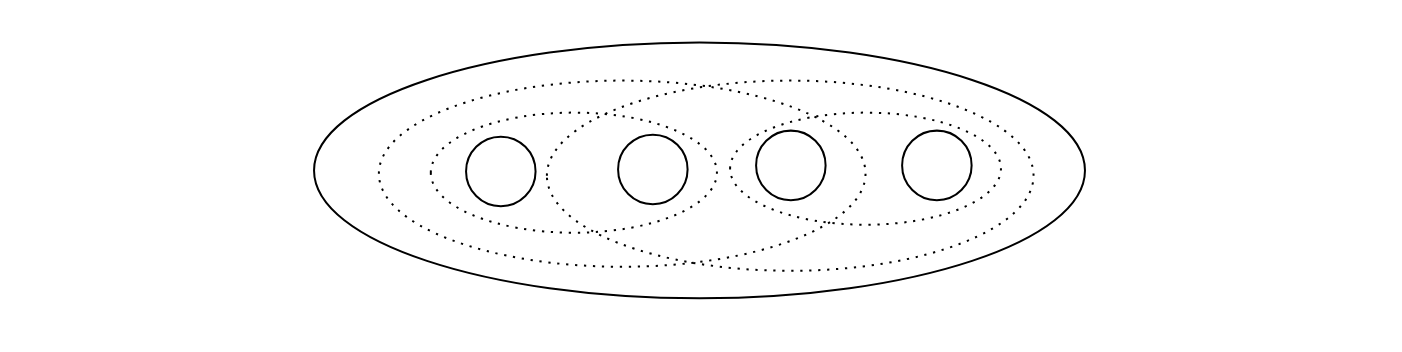}
\caption{Some incompressible tori in a Seifert fibered space with five singular fibers.\label{fig:IT}}
\end{figure}

 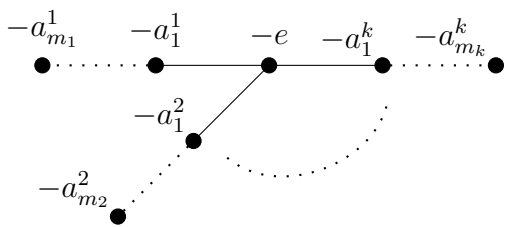
\begin{figure}
    \centering
    \begin{tikzpicture}[dot/.style = {circle, fill, minimum size=1pt, inner sep=0pt, outer sep=0pt}]
\tikzstyle{smallnode}=[circle, inner sep=0mm, outer sep=0mm, minimum size=2mm, draw=black, fill=black];

\node[smallnode, label={90:$-e$}] (0) at (0,0) {};
\node[smallnode, label={90:$-a_1^1$}] (a1) at (-1.5,0) {};
\node[smallnode, label={90:$-a_{m_1}^1$}] (a2) at (-3,0) {};

\node[smallnode, label={[label distance=-.15cm]150:$-a_1^2$}] (b1) at (-1,-1) {};
\node[smallnode, label={[label distance=-.15cm]150:$-a_{m_2}^2$}] (b2) at (-2,-2) {};

\node[smallnode, label={[label distance=-.15cm]150:$-a_1^k$}] (c1) at (1.5,0) {};
\node[smallnode, label={[label distance=-.15cm]150:$-a_{m_k}^k$}] (c2) at (3,0) {};

\node (a) at (-.7,-1.1) {};
\node (b) at (1.6,-.3) {};

\draw[-] (0) -- (a1);
\draw[loosely dotted, thick] (a1) -- (a2);

\draw[-] (0) -- (b1);
\draw[loosely dotted, thick] (b1) -- (b2);

\draw[-] (0) -- (c1);
\draw[loosely dotted, thick] (c1) -- (c2);

\draw[loosely dotted, thick] (a) [out=-40, in=-100] to (b);
\end{tikzpicture}
\caption{A star-shaped plumbing}
    \label{fig:starshaped}
\end{figure}

\end{document}